\documentclass[10pt]{article}

\usepackage[textsize=scriptsize]{todonotes}
\usepackage{changes}
\colorlet{Changes@Color}{blue}

\usepackage{amssymb}
\usepackage{amsmath}

\usepackage{amsthm}
\newtheorem{lemma}{Lemma}
\newtheorem{theorem}{Theorem}
\newtheorem{proposition}{Proposition}
\newtheorem{corollary}{Corollary}
\theoremstyle{definition}
\newtheorem{definition}{Definition}

\newtheorem{remark}{Remark}

\usepackage[titletoc,title]{appendix}

\usepackage[noblocks]{authblk}

\usepackage{graphicx}
\usepackage{caption}
\usepackage{subcaption}
\usepackage{epstopdf}
\usepackage{float}

\usepackage{booktabs}
\usepackage{microtype}

\usepackage{placeins}

\usepackage[hyperfootnotes=false]{hyperref}

\renewcommand{\div}{\operatorname{div}}
\newcommand{\norm}[1]{\left\| #1 \right\|}

\newcommand{\abs}[1]{\lvert#1\rvert}

\newcommand{\forget}[1]{}

\numberwithin{equation}{section}
\usepackage{color}

\def\div{\mathrm{div\,}}
\def\dim{d}
\def\Tn{\mathbb{T}^\dim}
\def\Hsav{H_{\mathrm{av}}^s}
\def\Hsavneg{H_{\mathrm{av}}^{-s}}
\def\Dsav{(-\Delta_{\mathrm{av}})^s}
\def\Dsavneg{(-\Delta_{\mathrm{av}})^{-s}}

\begin{document}

\title{A duality based approach to the minimizing total variation flow in the space $H^{-s}$}
\date{}

\author[1]{Yoshikazu Giga}
\author[2]{Monika Muszkieta}
\author[3]{Piotr Rybka}

\affil[1]{Graduate School of Mathematical Sciences, University of Tokyo 
\protect\\ Komaba 3-8-1, Tokyo 153-8914, Japan
\protect\\ e-mail: labgiga@ms.u-tokyo.ac.jp}
\affil[2]{Faculty of Pure and Applied Mathematics, Wroc{\l}aw University of Science and Technology
\protect\\ Wyb. Wyspia\'nskiego 27, 50-370 Wroc{\l}aw, Poland
\protect\\ e-mail: muszkieta@pwr.edu.pl}
\affil[3]{Institute of Applied Mathematics and Mechanics, Warsaw University
\protect\\ ul. Banacha 2, 02-097 Warsaw, Poland
\protect\\e-mail: rybka@mimuw.edu.pl}

\maketitle

\begin{abstract}
We consider a gradient flow of the total variation in a negative Sobolev space $H^{-s}$ $(0\leq s \leq 1)$ under the periodic boundary condition. If $s=0$, the flow is nothing but the classical total variation flow. If $s=1$, this is the fourth order total variation flow. We consider a convex variational problem which gives an implicit-time discrete scheme for the flow. By a duality based method, we give a simple numerical scheme to calculate this minimizing problem numerically and discuss convergence of a forward-backward splitting scheme. Several numerical experiments are given.
\end{abstract}



\section{Introduction}

During the last two decades, total variation models have became very popular in image processing and analysis. This interest has been initiated by the seminal paper of Rudin, Osher and Fatemi \cite{ROF1992}, where the authors proposed to minimize the functional
\begin{equation}
\label{rof}
\frac{1}{2\tau}\|u-f\|^2_{L^2(\Omega)}+  \int_{\Omega}|D u|
\end{equation}
to solve the problem of image denoising. Here, $\Omega\subset\mathbb{R}^2$ denotes the image domain, the function 
$f: \Omega\rightarrow \mathbb{R}$ represents a given gray-scale image and $\tau>0$ is a parameter.
The first term in (\ref{rof}) is called the fidelity term and it enforces a minimizer with respect to $u$ to be close to a given image $f$ in the sense of $L^2$-norm. 
The last term in (\ref{rof}) is the total variation of $u$ and it plays the role of regularization. 

The total variation model (\ref{rof}) was investigated in detail by Meyer \cite{Meyer2001} in context of the image decomposition problem. In this problem, it is assumed that a given image $f$ is a sum of two components, $u$ and $v$, where $u$ is a cartoon representation of $f$ and $v$ is an oscillatory component, composed of noise and texture. Meyer has observed that in the case when $f$ is a characteristic function, which is small with respect to a certain norm, the model (\ref{rof}) does not provide expected decomposition of $f$, since it treats such a function as oscillations. To overcome this inaccuracy, he proposed to consider a new model, replacing the $L^2$-norm in the fidelity term, by a weaker norm, which is more appropriate to deal with textures or oscillatory patterns (see \cite{Meyer2001} for details).
This change, however, introduced a practical difficulty in numerical computation of a minimizer, due to the impossibility of expressing
the associated Euler-Lagrange equation.
One of the first attempts to overcome this difficulty has been made by Osher, Sol\'e and Vese \cite{OshSolVes2003}. They proposed to approximate Meyer's model by its simplified version, given by
\begin{equation}
\label{osv}
\frac{1}{2\tau}\|u-f\|^2_{H^{-1}(\Omega)}+  \int_{\Omega}|D u|\,.
\end{equation}
Here $H^{-1}(\Omega)$ is the function space dual of $H^1_{0}(\Omega)$. 
Later, this model was studied in context of the decomposition problem by Elliott and Smitheman \cite{EllSmi2007}, \cite{EllSmi2008}. It has been also considered for other image processing task, like inpainting, by Sch\"onlieb \cite{Scho09} and Burger et. al \cite{BurHeSch09}.

The total variation models (\ref{rof}) and (\ref{osv}) and their interesting and successful applications in image processing became also the motivation for many authors to perform  rigorous analysis of properties of solutions to the corresponding total variation flows. In the case of the model (\ref{rof}), the corresponding total variation flow is formally given by the second order partial differential equation 
\begin{equation}
\label{eq_l2}
\frac{\partial u}{\partial t}=\div \left(\frac{\nabla u}{|\nabla u|}\right)\,.
\end{equation}
Whereas the total variation flow corresponding to the model (\ref{osv}) 
is provided by the fourth order partial differential equation
\begin{equation}
\label{eq_h1neg}
\frac{\partial u}{\partial t}=(-\Delta)\left[\div \left(\frac{\nabla u}{|\nabla u|}\right)\right]\,.
\end{equation}

We intend to give a few references related to (\ref{eq_l2}) and (\ref{eq_h1neg}) which is not at all exhaustive. Although the meaning of solutions for (\ref{eq_l2}), (\ref{eq_h1neg}) is unclear, it can be understood in a~naive way as the gradient flow of a lower semi-continuous convex functional in a Hilbert space. The well-posedness of (\ref{eq_l2}), (\ref{eq_h1neg}) has been established by the theory of maximal monotone operators initiated by Komura \cite{Ko} and developed by Brezis \cite{Ber73}. A behavior of solution for (\ref{eq_l2}) is rigorously studied, for example, in \cite{AnBaCaMa2001}
\cite{BeCaNo2002}. The book \cite{AnBaCaMa2001} contains several types of well-posedness results in different function spaces other than a Hilbert space $L^2$. Its anisotropic version was studied by  Mucha, Muszkieta and Rybka \cite{MucMusRyb15},  and \L{}asica, Moll and Mucha \cite{LMM} with intention to propose a numerical scheme for image denoising. Meanwhile, a viscosity approach for (\ref{eq_l2}) and its non-divergent generalization is established by M.-H. Giga, Y. Giga and N. Pozar \cite{GGP}. There is less literature on (\ref{eq_h1neg}) based on rigorous analysis. A characterization of the speed is given in Kashima \cite{K2004} (see \cite{Kas12} for a general dimension) and the extinction time estimate is given \cite{GigKoh11}, \cite{GigKurMat14}. It is also noted that a solution to (\ref{eq_h1neg}) may become discontinuous even if it is initially Lipschitz continuous \cite{GigGig10}, which is different from (\ref{eq_l2}). This type of equation is very important to model relaxation phenomena in materials science 
below the roughening temperature; see a recent paper by J.-G. Liu, J. Lu, D. Margetis and J. L. Marzuola \cite{LLMM} and papers therein. A numerical analysis for (\ref{eq_h1neg}) is given by Kohn and Versieux \cite{KV}.

In this paper, we generalize equations (\ref{eq_l2}) and (\ref{eq_h1neg}) and consider the total variation flow in the space $H^{-s}$ with $s\in [0,1]$. 
More precisely, we consider the $(2s+2)$-order parabolic differential equation
\begin{equation}
\label{eq_hsneg}
\frac{\partial u}{\partial t}=(-\Delta)^s \left[\div \left(\frac{\nabla u}{|\nabla u|}\right)\right]
\end{equation}
with periodic boundary conditions and initial data in $H^{-s}$. Our aim is to present a consistent approach to the construction of the minimizing total variation sequence in the space $H^{-s}$ and to discuss the problem of its convergence in an infinite dimensional space setting. Application of derived scheme will allow us to perform numerical experiments to observe evolution of solutions to the equation (\ref{eq_hsneg}) and their characteristic features with respect to different values of the index $s$. Such a close look at this problem may be the basis for further study on the considered evolution equation and its applications.
For example, our numerical experiment suggests that the solution may be discontinuous instantaneously for $s=1/2$ and $s=1$ for Lipschitz initial data. The case $s=1$ has been rigorously proved in [14]. We conjecture that such phenomena occur for all $s \in (0,1]$ excluding $s=0$ of course.

The well-posedness problem for (1.5) is similar to (1.3) and (1.4). Apparently, the characterization of speed of (1.5) was not derived before so we give its characterization. The proof given here is very close to the one for (1.3) given in \cite{AnBaCaMa2001}. We consider a semi-discretization of this equation with respect to a time variable. The scheme we consider here is an implicit-time discrete scheme. Such formula leads to recursive minimization of the functional functional
$$
\frac{1}{2\tau}\|u-f\|^2_{H^{-s}(\Omega)}+  \int_{\Omega}|D u|\,.
$$  
The non-differentiability of the total variation term 
caused a problem although the problem is convex. One way is to use Bergman's splitting method presented by Oberman, Osher, Takei and Tsai \cite{OOTT}. Here, we consider a dual formulation of this problem which is new at least for $s \in (0,1]$. We generate the minimizing dual sequence by the well-known splitting scheme and prove its convergence in $H^{-s}$. It seems that the proof of this setting seems to be new. We use the 
characterization of the subdifferential of the total variation in $H^{-s}$ to derive an explicit form of this sequence. It turns out that its form is very simple and convenient to compute an approximate minimizer. However, the proof of its convergence in $L^2$ can be carried out only in a finite dimensional space. Due to this limitation, we investigate the problem of convergence in ergodic sense.

    This paper is organized as follows. We recall $H^{-s}$ space and the total variation in Section 2. We give a characterization of the subdifferential of the total variation in $H^{-s}$ in Section 3. In Section 4, we give a way of semi-discritization, which is a recursive minimization problem of convex but non-differentiable functional.   
In Section 5, we formulate its dual problem. In Section 6, we discuss convergence of a forward-backward splitting scheme. In Section 7, we derive its explicit form. In Section 8, we discuss its ergodic convergence. In Section 9, we explain how to discretize the introduced scheme and provide an exact condition for a convergent problem in a finite dimensional space. In Section~10,  we present results of numerical experiments to illustrate evolution of solutions to considered total variation flows, showing their characteristic features with respect to different values of  index $s\in[0,1]$.

\section{Preliminaries}
To give a rigours interpretation of the equation (\ref{eq_hsneg}) with the periodic boundary conditions we first need to introduce preliminary definitions and notations. 

Let $\Tn$ be  a $\dim$-dimensional torus defined by
$$\Tn:=\prod_{i=1}^\dim \mathbb{R}/
\mathbb{Z}$$
For $s\in(0,1]$, we define by $H_{\text{av}}^{-s}(\Tn)$, the~space dual of
$$\Hsav(\Tn):=\left\{ u\in H^{s}(\Tn) \ : \ \int_{\Tn} u\,dx = 0 \right\}\,,$$
where $H^s(\Tn)$ is the standard fractional Sobolev space
$$H^s(\Tn)
:=\left\{u\in L^2(\Tn) \ : \  (1+|\xi|^{2s})^{1/2}\hat{u}(\xi)  \in L^2(\Tn) \right\}
$$
equipped with the norm
$$\norm{u}_{H^s(\Tn)}
:= \left(\int_{\Tn} (1+|\xi|^{2s})|\hat{u}(\xi)|^2\,d\xi\right)^{1/2}
$$
Here $\hat{u}$ denotes the Fourier transform of the function $u$, i.e.,
$$\hat{u}(\xi) = \int_{\Tn} u(x) e^{-2\pi i x \xi}\,dx\,.$$
We define the operator $\Dsav:\Hsav(\Tn)\rightarrow \Hsavneg(\Tn)$ by
\begin{equation}
\widehat{\Dsav u}(\xi):=|\xi|^{2s}\widehat{u}(\xi)\,,
\end{equation}
Then the inner product in $\Hsav$ is defined by
$$(u,v)_{\Hsav}
:=\int_{\Tn} |\xi|^{2s}\, \hat{u}(\xi)\,\hat{v}(\xi)\,d\xi = \langle \widehat{\Dsav u}, \hat{v}\rangle = 
\langle \Dsav u, v\rangle$$
for all $u,v \in \Hsav(\Tn)$. The last equality in the above definition follows Parseval's theorem. 

Since the operator $\Dsav$ is an isometry, therefore the inner product in $\Hsavneg(\Tn)$ is given by
$$(f,g)_{\Hsavneg}:=\langle \widehat{\Dsavneg f}, \hat{g}\rangle = \langle \Dsavneg f, g\rangle$$
for all $f$, $g\in H_{\text{av}}^{-s}(\Tn)$. 

Now, we introduce the definition of a total variation of the function $u$. It is given by
$$\int_{\Tn} \abs{D u}
:= \sup\left\{\int_{\Tn} u\, \div z \,dx \ : \ z\in C_0^1(\Tn; \mathbb{R}^d),\  \norm{z}_{\infty}\leq 1 \right\}\,,$$
where for a vector field $z(x)=(z_1(x),z_2(x))$, the norm $\norm{\cdot}_{\infty}$ is defined by $\norm{z}_{\infty}:=\sup_x |z(x)|$ and $|\cdot|$ denotes the standard Euclidean norm. The subspace of periodic functions $u\in L^1(\Tn)$ such that the total variation of $u$ is finite is denoted by $BV(\Tn)$.

\section{A characterization of the subdifferential}
We define the functional $\Phi$ on $L^2(\Tn)$ by
\begin{equation}
\label{defPhi}
 \Phi(u):=\left\{\begin{array}{ll} 
 \int_{\Tn} \abs{D u} & \text{ if } u\in BV(\Tn)\cap \Hsavneg(\Tn)\,,\\[0.1cm]
+\infty & \text{ otherwise}\,.
 \end{array}\right.
\end{equation}
We also need to define the associated dual functional
\begin{equation}
\label{defPhitilde}
\tilde{\Phi}(v) := \sup_w \left\{\frac{(w,v)_{\Hsavneg}}{ \Phi(w)} \ : \ w\in \Hsavneg(\Tn)\setminus\{0\}\right\}\,.
\end{equation}

Further, we need to introduce some standard results.
We assume that $H$ is a normed space and $H^\ast$ is its dual space.
\begin{lemma}
\label{lem1a}
Let $\Phi$, $\Psi: H \rightarrow [0,\infty]$.
If $\Phi\leq\Psi$ then $\tilde{\Psi}\leq \tilde{\Phi}$.
\end{lemma}
\begin{proof}
See proof of \cite[Lemma 1.5]{AndMazCas04}
\end{proof}

\begin{lemma}
\label{lem1b}
Suppose $\Phi$ is convex, lower semi-continuous and positively homogeneous of degree one, then
$\tilde{\tilde{\Phi}}(u) =\Phi(u)$.
\end{lemma}
\begin{proof}
See proof of \cite[Proposition 1.6]{AndMazCas04}
\end{proof}

\begin{definition}
\label{def:subdifferential}
Let $\Phi$ be a proper lower-semicontinuous, convex functional on $H$, endowed with the inner product  $( \cdot,\cdot)_H$. The subdifferential of $\Phi$ at $u\in H$ is the set
$$\partial_H \Phi(u) := \{v\in H^\ast \ : \ \Phi(w) \geq  \Phi(u) + ( v ,w-u)_H \quad 
\forall w\in H\}.$$
\end{definition}

\begin{lemma}
\label{lem1}
Suppose $\Phi$ is convex, lower semicontinuous, nonnegative, and positively homogeneous of degree one. Then $v\in \partial_{\Hsavneg} \Phi(u)$ if and only if $\tilde{\Phi}_{\Hsavneg}(v)\leq 1$ and $( u,v)_{\Hsavneg} =\Phi(u)$.
\end{lemma}
\begin{proof}
See proof of \cite[Theorem 1.8]{AndMazCas04}
\end{proof}

Further we will characterize the subdifferential of $\Phi$ in $\Hsavneg(\Tn)$. 
We define
\begin{equation}
X(\Tn):=\{z\in L^{\infty}(\Tn;\mathbb{R}^n) \ : \ \div z\in \Hsav(\Tn)\}
\end{equation}

In the next lemma, we will show that $\tilde{\Phi}(u)=\Psi(u)$, where $\Psi$ is ginve by (\ref{defPsi}). This will enable us to characterize the subdifferential of  $\Phi$. In the proof, we follow results presented in \cite[Lemma 8.5, Proposition 8.6]{GigKurMat14} and \cite{GigKoh11}.

\begin{lemma}
\label{lem2}
Let $\Psi$ be the functional defined by
\begin{equation}
\label{defPsi}
\Psi(w):=\inf \{\norm{z}_{L^{\infty}(\Tn)} \ : \ z\in X(\Tn),\ w=-\Dsav \div z \}\,.
\end{equation}
Then $\Psi(w)=\tilde{\Phi}(w)$ for $w\in \Hsavneg(\Tn)$, where $\tilde{\Phi}$ is the functional dual to $\Phi$.
\end{lemma}
\begin{proof}
Let $w\in\Hsavneg(\Tn)$. First, we show $\tilde{\Phi}(w)\leq\Psi(w)$. If $\Psi(w)$ is infinite, then the inequality is obvious. Thus, assume that $\Psi(w)<\infty$, then there exists $z\in X(\Tn)$ such that
$$w=-\Dsav\div z \  , \  \norm{z}_{L^{\infty}(\Tn)}=\Psi(u)\,.$$
Take $z^{\varepsilon}=z\ast\rho^{\varepsilon}$, where $\rho$ is a molifier. In \cite{Kas12}, Kashima showed that
$\norm{z^{\varepsilon}}_{L^{\infty}(\Tn)}\leq \norm{z}_{L^{\infty}(\Tn)}$
and
$\div z^{\varepsilon}\rightarrow\div z$ in $\Hsav(\Tn)$
as $\varepsilon\rightarrow 0$.
Then using this result, we obtain
\begin{equation}
\begin{split}
(w,v)_{\Hsavneg(\Tn)}
&=\langle -\div z,v\rangle 
= \lim_{\varepsilon\rightarrow 0}\langle -\div z^{\varepsilon},v\rangle\\
&\leq \limsup_{\varepsilon\rightarrow 0} \norm{z^{\varepsilon}}_{L^{\infty}(\Tn)} \Phi(v)
\leq \norm{z}_{L^{\infty}(\Tn)} \Phi(v)
\end{split}
\end{equation}
for all $v\in \Hsavneg(\Tn)$. Therefore, we get that
$$\tilde{\Phi}(w)=\sup \left\{ \frac{(w,v)_{\Hsavneg}}{\Phi(v)} \ : \ v\in \Hsavneg(\Tn)\setminus\{0\}\right\}
\leq \norm{z}_{L^{\infty}(\Tn)} = \Psi(w)\,,$$
which is the desired inequality.
Now we prove that $\tilde{\Phi}(w)\geq\Psi(w)$. 
To do so, we first note that if we prove $\Phi(w)\leq\tilde{\Psi}(w)$ then by Lemma \ref{lem1a} we get that $\tilde{\Phi}(w)\geq\tilde{\tilde{\Psi}}(w)$. Furthermore, since $\Psi$ is convex, lower-semicontinuous, and positively homogeneous of degree one, we get that $\Psi=\tilde{\tilde{\Psi}}$. This will imply the desired inequality.

Therefore, we need to prove that $\Phi(w)\leq\tilde{\Psi}(w)$. By the definition of $\Psi$, we get that
\begin{equation}
\begin{split}
\tilde{\Psi}(w)&=\sup \left\{ \frac{(w,v)_{\Hsavneg}}{\Psi(v)} \ : \ v\in \Hsavneg(\Tn)\setminus\{0\}\right\}\\
&\geq  
\sup \left\{ \frac{\langle w,\div z\rangle}{\norm{z}_{L^\infty(\Tn)}} \ : \ z\in C^{\infty}(\Tn;\mathbb{R}^n)\setminus\{0\}\right\}
= \Phi(w)\,.
\end{split}
\end{equation}
\end{proof}

\begin{theorem}
\label{thm1}
Assume that $u\in \Hsavneg(\Tn)$ is such that $\Phi(u)<\infty$. We denote by $\partial_{\Hsavneg}\Phi$  the subdifferential of $\Phi$ with respect to $\Hsavneg(\Tn)$-topology. Then, $v\in \partial_{H^{-s}} \Phi(u)$ if and only if there exists $z\in X(\Tn)$ such that
\begin{equation}
\left\{\begin{array}{lcl} 
v = -\Dsav \div z\,,\\[0.2cm]
\norm{z}_{L^{\infty}(\Tn)}\leq 1\,,\\[0.2cm]
(u, -\Dsav\div z)_{\Hsavneg(\Tn)}=\int_{\Tn}|D u|\,.
\end{array}\right.
\end{equation}
\end{theorem}
\begin{proof}
By Lemma \ref{lem1}, we see that 
$v\in \partial_{\Hsavneg} \Phi(u)$ if and only if $\tilde{\Phi}_{\Hsavneg}(v)\leq 1$ and $(u,v)_{\Hsavneg(\Tn)}=\Phi(u)$. Moreover, $\Psi=\tilde{\Phi}$ holds by Lemma \ref{lem2}. These conditions are equivalent to
$$\Psi(v)\leq 1 \ \text{ and } (u,v)_{\Hsavneg(\Tn)}=\int_{\Tn} |D u|\,.$$
\end{proof}

Theorem \ref{thm1} gives a rigorous interpretation of the equation (\ref{eq_hsneg}) as
\begin{equation}
\label{eq_rigorous}
\left\{\begin{array}{lll} 
\dfrac{d u}{d t}(t)\in-\partial_{\Hsavneg}  \Phi(u(t))& \text{ in } \Hsavneg(\Tn) \text{ for a.e. } t\in (0,\infty)\,,\\[0.2cm]
u(0)=u_0 & \text{ in }  \Hsavneg(\Tn)\,. &
\end{array}\right.
\end{equation}

The result concerning the existence and uniqueness of a solution to  system (\ref{eq_rigorous}) is given in the theorem below.

\begin{theorem}
Let $\mathcal{A}(u) := \partial_{\Hsavneg} \Phi(u)$ and suppose that $u_0\in D(\mathcal{A})$. 
Then, there exists a unique function $u:[0,\infty)\rightarrow \Hsavneg(\Tn)$ such that:
\begin{enumerate}
\setlength\itemsep{0.1cm}
\item[$\mathrm{(1)}$] for all $t>0$ we have that $u(t)\in D(A)$,
\item[$\mathrm{(2)}$] $\frac{du}{dt}\in L^\infty(0,\infty, L^2(\Tn))$ and 
$\norm{\frac{du}{dt}}_{\Hsavneg(\Tn)}
\leq \norm{\mathcal{A}^0 (u_0)}_{\Hsavneg(\Tn)}$,
\item[$\mathrm{(3)}$]
 $\frac{du}{dt}\in -\mathcal{A}(u(t))$ a.e. on $(0,\infty)$
\item[$\mathrm{(4)}$]
$u(0)=u_0$.
\end{enumerate}
\end{theorem}
\begin{proof}
The proof of this theorem can be found in Brezis \cite{Ber73}.
\end{proof}

\section{A semi-discretization}

In order to construct the scheme to solve the equation (\ref{eq_rigorous}), we introduce its semi-discretization with respect to the time variable $t$. That is, we consider the finite set of $n+1$ equidistant points 
$\{t_i = i\tau \ : \ i=0,\ldots,n \text{ and }\tau = T/n\}$ in the interval $[0,T]$.
For $i=1,\ldots,n$, we define inductively a  $u_{\tau}(t_i)$ as a solution of 
\begin{equation}
\label{semidiscrete}
\dfrac{u_{\tau}(t_i)-u_{\tau}(t_{i-1})}{\tau}\in-\partial_{\Hsavneg} \Phi(u_{\tau}(t_i)) \quad \text{ in }\Hsavneg(\Tn)\,.
\end{equation} 
It is well known that if $\mathcal{A} = \partial_{\Hsavneg} \Phi$ is monotone, then the resolvent $\mathcal{J}_\tau := (I+\tau \mathcal{A})^{-1}$ is non-expansive, which implies that the implicit scheme (\ref{semidiscrete}) is stable and we have
$$u_\tau(t_i) = \mathcal{J}_{\tau}^i u_0\,.$$
This result has been justified in a very general setting
by Crandall-Liggett \cite{CraLig71}: 
\begin{theorem}
For all $u_0\in \overline{D(\mathcal{A})}$ and $t>0$ we have that
$$\lim_{\tau \rightarrow 0,\ t_i\rightarrow t} \mathcal{J}_{\tau}^i u_0 = S(t) u_0\,,$$
where $S(t)$ is the semigroup generated by $-\mathcal{A}$. The convergence is uniform on compact intervals of $[0,\infty)$.
\end{theorem}

The original paper by Crandall-Liggett 
contains also an error estimate, which is not optimal. The optimal one
is of order $O(\tau^2)$ and it has been derived by Rulla \cite[Theorem 4]{Rul96}.

Hence, to solve (\ref{eq_rigorous}), we need to know how to find a solution to the one iteration of the scheme (\ref{semidiscrete}). First, we observe that the equation (\ref{semidiscrete}) for $u_{\tau}(t_i)$ is the optimality condition for the minimization problem
\begin{equation}
\label{primal}
\inf_{u}\left(\frac{1}{2\tau}\|u-u_{\tau}(t_{i-1})\|^2_{H_{\text{av}}^{-s}} + \Phi(u)\right)\,.
\end{equation}

The main difficulty in the construction of a minimizing sequence for such kind of problems is caused by the~lack of differentiability of the total variation term. The commonly used approach to overcome this difficulty consists in considering the dual formulation of (\ref{primal}). The~first work related to image processing application where this approach has been used and where the proof of the convergence has been provided, was the paper of Chambolle \cite{Chambolle2004}.
Nowadays, there is a variety of efficient schemes that one can apply to solve the dual formulation of (\ref{primal}). Among others, we mention here papers of Chambolle and Pock \cite{Pock2011} and Beck and Teboulle \cite{Beck2009}.
We also refer to papers of Aujol \cite{Aujol2009} and Weiss et al. \cite{Weiss2009}, where the authors adapt the existing methods for the purpose of the total variation minimization. In this work, we rather aim to present a consistent approach for construction of the minimizing sequence and to discuss its convergence in an infinite dimensional space. For this purpose, we consider the well-known splitting scheme introduced by Lions and Mercier~\cite{LioMer79}, which in fact is a variant of the classical Douglas-Rachford algorithm. 

\section{A dual problem}

Let $C\subset H$ be a non-empty convex and closed set. We define the indicator function of a set $C$ by
\begin{equation}
\label{defchi}
\chi_C(u) := \left\{\begin{array}{cl} 
 0 & \text{ if } u\in C\,,\\[0.1cm]
+\infty & \text{ otherwise}\,.
\end{array}\right.
\end{equation}

To drive the dual problem we will need the following result
\begin{lemma}
\label{lem7}
If we consider  $\Phi$ given by (\ref{defPhi}) as a functional definwed on $\Hsavneg$. Then,
the convex conjugate of 
$\Phi$ 
is given by
$\Phi^\ast(v)=\chi_{K}(v)$,
where $K$ is the closure of the set $$\{v\in \mathcal{D}'(\Tn) \ : \ v=-\Dsav \div z, \  z\in \mathcal{D}(\Tn),\ \norm{z}_{\infty}\leq 1\}$$
with respect to the strong $\Hsavneg(\Tn)$-topology.
\end{lemma}
\begin{proof}
Let $u\in BV(\Tn)\cap\Hsavneg(\Tn)$, then 
the convex conjugate of  function $\chi_{K}$ is given by
\begin{equation}
\begin{split}
\chi_{K}^\ast(u)
&=\sup_{v}\{(v,u)_{\Hsavneg} \ : \ v\in K\}\\
&=\sup_{z}\{\langle u,\div z \rangle \ : \ z\in \mathcal{D}(\Tn),\ \norm{z}_{\infty}\leq 1\}=\Phi(u)\,.
\end{split}
\end{equation}
Since the set $K$ is convex, the function $\chi_{K}$ is convex. Moreover, since $\chi_{K}$ is lower semi-continuous, then it follows from \cite[Proposition I.4.1]{Ekeland1999}
 
that $\Phi^\ast(v)=\chi_{K}^{\ast\ast}(v)=\chi_{K}(v)$.
\end{proof}
\begin{proposition}
Let $f\in \Hsavneg(\Tn)$ be a given function, then the problem 
\begin{equation}
\label{primal_problem}
\inf_{u\in \Hsavneg(\Tn)} \left\{\frac{1}{2\tau}\norm{u-f}_{\Hsavneg}^2+\Phi(u) \right\}
\end{equation}
is equivalent with
\begin{equation}
\label{dual_problem_v}
\inf_{v\in K} \left\{\frac{1}{2\tau}\norm{\tau v- f}_{\Hsavneg}^2\right\}\,.
\end{equation}
Moreover, the~solution~$u$ of   (\ref{primal_problem}) is
associated with the solution $v$ of  (\ref{dual_problem_v}) by the relation 
\begin{equation}
\label{relation_uv}
u=f-\tau v\,.
\end{equation}
\end{proposition}
\begin{proof}

Using the standard result of convex analysis (see, e.g, \cite[Corollary I.5.2]{Ekeland1999}), we get that the~Euler-Lagrange equation 
\begin{equation}
\label{inc1}
\frac{f-u}{\tau}\in\partial_{\Hsavneg} \Phi(u)
\end{equation}
associated with the problem (\ref{primal_problem}) 
is equivalent to
\begin{equation}
\label{inc2}
u\in \partial_{\Hsavneg} \Phi^{\ast}\left(\frac{f-u}{\tau}\right)\,,
\end{equation}
where $\Phi^{\ast}$ is the convex conjugate functional of $\Phi$ in $\Hsavneg(\Tn)$. Setting $\tau v= f-u$, we rewrite equation~(\ref{inc2}) as
\begin{equation}
\label{inc3}
 f-\tau v\in \partial_{\Hsavneg} \Phi^{\ast}(v)\,,
\end{equation}
and we note that this is the Euler-Lagrange equation of the functional
$$\dfrac{1}{2\tau}\norm{\tau v - f}^2_{\Hsavneg}+ \Phi^{\ast}(v)\,.$$
Using Lemma \ref{lem7}, we conclude that the dual problem to (\ref{primal_problem}) is
given by (\ref{dual_problem_v}) and the relation (\ref{relation_uv}) holds.
\end{proof}

\begin{corollary}
\label{cor:projection}
The solution of problem (\ref{primal_problem}) satisfies $u = f-\tau P_K^{\Hsavneg}(f/\tau)$,
where $P_K^{\Hsavneg}$ denotes 
the orthogonal projection on the set $K$ with respect to the inner product in $\Hsavneg$ .
\end{corollary}

\begin{remark}
Using the characterization of $v\in \partial_{H^{-s}} \Phi(u)$ provided in Theorem~\ref{thm1}, we can rewrite the dual problem (\ref{dual_problem_v}) as
\begin{equation}
\label{dual_problem_z}
\inf_{z\in Z} \left\{\frac{1}{2\tau}\norm{\tau \Dsav\div z + f}_{\Hsavneg}^2\right\}\,,
\end{equation}
where $Z$ is the closure of the set
$$\left\{z\in \mathcal{D}(\Tn) \ : \Dsav\div z \in \mathcal{D}'(\Tn)\,, \norm{z}_{\infty}\leq 1 \ \right\},$$
with respect to the strong  $L^2$-topology.
\end{remark}

\section{Convergence of a forward-backward splitting scheme}

Let us define the functional $J$ on $\Hsavneg(\Tn)$ by
\begin{equation*}
 J(v):=\left\{\begin{array}{ll} 
\dfrac{1}{2\tau}\norm{\tau v-f}_{\Hsavneg}^2 & \text{ if } v\in K\,,\\[0.2cm]
+\infty & \text{ otherwise}\,.
 \end{array}\right.
 \nonumber
\end{equation*}
Then the dual problem (\ref{dual_problem_v}) can be equivalently written as
$$\inf_{v\in \Hsavneg(\Tn)} \left( J(v) + \Phi^\ast(v) \right)\,.$$
In order to construct a  minimizing sequence $\{v^k\}$ for the above problem, we consider the forward-backward splitting scheme introduced by Lions and Mercier \cite{LioMer79}, given by
\begin{equation}
\label{scheme1}
\left\{\begin{array}{l} 
u^k \in -\partial_{\Hsavneg} J(v^k)\,, \\[0.2cm]
v^{k+1} = (I+\lambda \partial_{\Hsavneg} \Phi^\ast)^{-1}(v^k +\lambda u^k)\,.
\end{array}\right.
\end{equation}
\begin{remark}
Application of the scheme (\ref{scheme1}) requires that 
$$\partial_{\Hsavneg}(J + \Phi^\ast)(v) = \partial_{\Hsavneg}J(v) + \partial_{\Hsavneg}\Phi^\ast(v)$$ 
for all $v\in \Hsavneg(\Tn)$. This is ensured by  \cite[Proposition I.5.6]{Ekeland1999}, because $\hbox{int}\,(D(J))\cap D(\Phi^\ast)\neq \emptyset$.
\end{remark}

\begin{remark}
\label{rem_yosida}
Let $v\in \Hsavneg$, then by
Moreau's identity
$$v = (I+\lambda \partial_{\Hsavneg} \Phi^\ast)^{-1}(v) 
  +\lambda\left(I+1/\lambda\, \partial_{\Hsavneg} \Phi\right)^{-1}\left(v/\lambda\right)\,,$$
we obtain that the scheme (\ref{scheme1}) is equivalent to
\begin{equation}
\left\{\begin{array}{l} 
u^k \in -\partial_{\Hsavneg} J(v^k)\,, \\[0.2cm]
v^{k+1} = H_{1/\lambda}(v^k/\lambda + u^k)\,,
\end{array}\right.
\nonumber
\end{equation}
where $H_{1/\lambda}$ denotes the Yosida approximation of the operator  $\mathcal{A} = \partial_{\Hsavneg} \Phi$. It is well known that $H_{1/\lambda}$ converges as $\lambda\rightarrow \infty$ to the minimal selection $\mathcal{A}_0$ 
of~$\mathcal{A}$.
\end{remark}

In the proof of convergence of the sequence $\{v^k\}$ we will need the lemma below:
\begin{lemma}
\label{lem:ineq}
For $v\in K$ we have that if $u\in \partial_{\Hsavneg} \Phi^\ast(v)$, then
$(u,v-w)_{\Hsavneg}\geq 0$ for all $w\in K$.
\end{lemma}
\begin{proof}
By the definition of the subdifferential $\partial_{\Hsavneg} \Phi^\ast(v)$, we have that
$u\in \partial_{\Hsavneg} \Phi^\ast(v)$ if and only if
$\Phi^\ast(w) \geq  \Phi^\ast (v) + (u,w-v)_{\Hsavneg}$
for all $w\in \Hsavneg(\Tn)$.
Since $\Phi^\ast$ is the indicator function of the set $K$, we have $(u,v-w)_{\Hsavneg}\geq 0$ for all $w\in K$.
\end{proof}

\begin{proposition}
\label{prop_converg_vk}
Assume that $0<\lambda\tau<2$, then $v^k\rightharpoonup v^\ast$ and $u^k\rightharpoonup u^\ast$ in $\Hsavneg$, where $v^\ast\in K$ is such that $v^\ast\in \partial_{\Hsavneg} \Phi(u^\ast)$ and $u^\ast = f-\tau v^\ast$.
\end{proposition}
\begin{proof}
Assume that $v^{k}$, $v^{k+1} \in K$. After simple calculations, we obtain
\begin{equation}
\label{proof_calc11}
\begin{split}
J(v^{k+1})-J(v^k)&= 
\frac{1}{2\tau}\norm{\tau v^{k+1}-f}_{\Hsavneg}^2
-\frac{1}{2\tau}\norm{\tau v^{k}-f}_{\Hsavneg}^2\\
& = \frac{\tau}{2}\norm{v^{k+1}-v^k}_{\Hsavneg}^2
 + ( f-\tau v^k, v^{k+1}-v^k  )_{\Hsavneg}\,.\\
\end{split}
\end{equation}
From Lemma \ref{lem:ineq}, it follows that
\begin{equation}
\label{aaa}
\frac{1}{\lambda}(v^k-v^{k+1}+\lambda u^k,w-v^{k+1})_{\Hsavneg}\leq 0
\nonumber
\end{equation}
for all $w\in K$. In particular, taking $w=v^k$ and recalling that $u^k = f-\tau v^k$, we get
\begin{equation}
\label{proof_calc2}
(\tau v^k-f,v^{k}-v^{k+1})_{\Hsavneg}\leq -\frac{1}{\lambda}\norm{v^k-v^{k+1}}_{\Hsavneg}^2\,.
\nonumber
\end{equation}
Taking into account this fact in (\ref{proof_calc11}), we obtain
\begin{equation} 
\label{ineq2}
\begin{split}
J(v^{k+1})-J(v^k)
&\leq \frac{\tau}{2}\norm{v^{k+1}-v^k}_{\Hsavneg}^2-\frac{1}{\lambda}\norm{v^k-v^{k+1}}_{\Hsavneg}^2\\
&\leq \left(\frac{\tau}{2}-\frac{1}{\lambda}\right) \norm{v^k-v^{k+1}}_{\Hsavneg}^2\,.
\end{split}
\end{equation}
Therefore, taking $\lambda$ so that $0<\lambda\tau<2$, we get 
$J(v^{k+1})<J(v^k)$ as long as $v^{k+1}\neq v^k$. This implies that the sequence $\{J(v^k)\}$ is convergent.

Let $m = \lim_{k\rightarrow \infty} J(v^k)$. Then, obviously $\lim_{k\rightarrow \infty} J(v^{k+1}) = m$, 
and we have
$$\lim_{k\rightarrow \infty} (J(v^{k})-J(v^{k+1})) = 0\,.$$
Since $\lambda \tau<2$ by the assumption, we have
\begin{equation} 
\label{ineq2a}
J(v^{k})-J(v^{k+1})\geq \left(\frac1\tau -\frac\tau2\right) \norm{v^k-v^{k+1}}_{\Hsavneg}^2\,.
\end{equation}
Then, passing to the limit as $k\rightarrow \infty$, we obtain
\begin{equation} 
\label{limvk}
\lim_{k\rightarrow \infty} \|v^{k+1}-v^k\|_{\Hsavneg} = 0\,.
\end{equation}

We notice that boundedness of 
the sequence $\{J(v^k)\}$ and  
the definition of $J$ imply 
that the sequence $\{v^k\}$ is bounded in $\Hsavneg$. Moreover, since $u^k = f-\tau v^k$, the sequence $\{u^k\}$ is also bounded in $\Hsavneg$ .
Then let denote by $v^\ast$ a limit of the convergent subsequence $\{v^{k_j}\}$ of $\{v^{k}\}$ and
by $v^{\ast\ast}$ a limit of the convergent subsequence $\{v^{k_j+1}\}$ of $\{v^{k+1}\}$. The fact (\ref{limvk}) implies that $v^\ast=v^{\ast\ast}$. It remains to show that $v^\ast\in \partial_{\Hsavneg}\Phi(u^\ast)$, where $u^\ast=f-\tau v^\ast$ is a limit of the convergent 
subsequence $\{u^{k_j}\}$ of $\{u^k\}$.

The monotonicity of the operator $\partial_{\Hsavneg} \Phi^\ast$ implies that
$$(\partial_{\Hsavneg} \Phi^\ast(v^{k+1})-\partial_{\Hsavneg} \Phi^\ast(w),v^{k+1}-w)_{\Hsavneg}\geq 0$$
for all $w\in \Hsavneg(\Tn)$.
Using the second line of the scheme (\ref{scheme1}) we can rewrite the above inequality as
\begin{equation}
\label{ineq_star}
\frac{1}{\lambda}( v^k-v^{k+1},v^{k+1}-w)_{\Hsavneg}+(u^k-\partial_{\Hsavneg} \Phi^\ast(w),v^{k+1}-w)_{\Hsavneg}\geq 0\,.
\end{equation}

The first term in (\ref{ineq_star}) goes to zero as $k\rightarrow \infty$, ecause it can be estimated by
$$
\frac{1}{\lambda} \| v^{k+1} - v^k \|_{\Hsavneg} \| v^{k+1} -w \|_{\Hsavneg}.
$$
The first factor above goes to zero, while the second one is bounded.

We rewrite the second term in (\ref{ineq_star}),as follows,
$$
(u^k, v^{k+1} - w)_{\Hsavneg}
- (\partial_{\Hsavneg}\Phi^*(w), v^{k+1} - w)_{\Hsavneg}.
$$
The weak convergence of $v^{k+1}$ implies
$$
\lim_{k\to\infty} (\partial_{\Hsavneg}\Phi^*(w), v^{k+1} - w)_{\Hsavneg} =
(\partial_{\Hsavneg}\Phi^*(w), v^\ast- w)_{\Hsavneg}.
$$
Moreover,
\begin{equation}
\begin{split}
(u^k, v^{k+1} &- w)_{\Hsavneg} 
=(f-\tau v^k,v^{k+1} - w)_{\Hsavneg}\\
&= (f- \tau v^{k+1}, v^{k+1} - w)_{\Hsavneg} +
\tau(v^{k+1} -v^k,v^{k+1} - w)_{\Hsavneg}\,.
\end{split}
\end{equation}
We notice that due to (\ref{limvk}),the term
$(v^{k+1} -v^k,v^{k+1} - w)_{\Hsavneg}$ goes to zero.
In addition,
$$
\limsup_{k\to \infty} \left( - \| v^k \|^2_{\Hsavneg} \right) =
- \liminf_{k\to \infty}  \| v^k \|^2_{\Hsavneg} \le
- \| v^\ast \|^2_{\Hsavneg}.
$$
If we combine these pieces of information, then passing to the limit in (\ref{ineq_star}) with $k\rightarrow \infty$, we obtain
$$(u^\ast-\partial_{\Hsavneg} \Phi^\ast(w),v^\ast-w)_{\Hsavneg}\geq 0\,.$$

Since the functional $\Phi^\ast$ is convex, proper and lower semi-continuous, its sub-differential 
$\partial_{\Hsavneg} \Phi^\ast$ is a maximal monotone operator. This implies that $u^\ast\in \partial_{\Hsavneg}\Phi^\ast(v^\ast)$, what
by is equivalent to $v^\ast\in \partial_{\Hsavneg}\Phi(u^\ast)$, due to \cite[Corollary I.5.2]{Ekeland1999}.
\end{proof}

\section{An explicit form of a scheme}

Since the nonlinear projection, given in Corollary \ref{cor:projection}, is difficult to compute in practice, we use the characterization of subdifferential $\partial_{\Hsavneg} \Phi$, provided in Theorem \ref{thm1}, to construct an explicit scheme for minimizing sequence of the dual problem~(\ref{dual_problem_z}). More precisely, we apply the same forward-backward scheme as in the previous section 
to generate the sequence $\{z^k\}$ minimizing (\ref{dual_problem_z}). This sequence is related to $\{v^k\}$ by formula $v^k = -\tau \Dsav\div z^k$. 

In order to proceed, we need to introduce definitions of functionals $F$ and $G$ on $L^2(\Tn,\mathbb{R}^d)$. They are given by
\begin{equation}
\label{defF}
 F(z):=\left\{\begin{array}{ll} 
\dfrac{1}{2\tau}\norm{\tau \Dsav\div z + f}_{\Hsavneg}^2 & \text{ if } z\in Z(\Tn)\,,\\[0.2cm]
+ \infty & \text{ otherwise}\,.
 \end{array}\right.
\end{equation}
and 
\begin{equation}
\label{defG}
G(z) := \left\{\begin{array}{ll} 
 0 & \text{ if } z\in Z(\Tn)\,,\\[0.1cm]
+\infty & \text{ otherwise}\,.
\end{array}\right.
\end{equation}
Then, we notice that the dual problem (\ref{dual_problem_z}) is equivalent with
\begin{equation}
\label{min_problem_z_cont}
\inf_{z\in L^2(\Tn,\mathbb{R}^d)} \left( F(z) + G(z) \right)
\end{equation}

In order to find $z^\ast$ such that $0\in \partial_{L^2} (F(z^\ast) + G(z^\ast))$, we consider the splitting scheme:
\begin{equation}
\label{scheme3}
\left\{\begin{array}{l} 
w^k \in \partial_{L^2} F(z^k)\,, \\[0.2cm]
z^{k+1} = (I+\lambda \partial_{L^2} G)^{-1}(z^k - \lambda w^k)\,.
\end{array}\right.
\end{equation}
For practical implementation of this scheme, it remains to find its explicit formula.

\begin{lemma}
\label{scheme_lemma1}
Assume that $z\in L^2(\Tn)$ is such that $F(z)<\infty$. 
Then we have $\partial_{L^2} F(z)=-\nabla (\tau \Dsav \div z + f)$.
\end{lemma}
\begin{proof}
Since the functional $F$ is differentiable, we have that $\partial_{L^2} F(z)=\{\nabla F(z)\}$. By direct calculations, we obtain
\begin{equation}
\begin{split}
&F(z+\epsilon\eta)-F(z) = \dfrac{1}{2\tau}\norm{\tau\Dsav\div (z+\epsilon\eta) + f}_{\Hsavneg}^2 - \dfrac{1}{2\tau}\norm{\tau\Dsav\div z + f}_{\Hsavneg}^2 \\
&\qquad= \frac{1}{2\tau}(\tau \Dsav\div (2 z + \epsilon \eta) + 2f, \tau \Dsav\div (\epsilon\eta))_{\Hsavneg}\\
&\qquad= \epsilon(\tau \Dsav\div z + f, \Dsav\div \eta)_{\Hsavneg}
+\frac{\epsilon^2}{2}\norm{\tau \Dsav\div \eta}_{\Hsavneg}^2\,.
\end{split}
\nonumber
\end{equation}
for any $\eta\in L^2(\Tn)$ and $\epsilon>0$.
Thus, using the definition of the inner product in $\Hsavneg$ and the integration by parts formula, we get
\begin{equation}
\begin{split}
\lim_{\epsilon\rightarrow 0}\dfrac{1}{\epsilon} 
\left(F(z+\epsilon\eta)-F(z)\right) 
&= \langle \tau\Dsav\div z + f, \div\eta\rangle\\
&= -\langle \nabla(\tau \Dsav\div z + f), \eta  \rangle\,,
\end{split}
\end{equation}
what yields the desired result.
\end{proof}

\begin{lemma}
\label{scheme_lemma2}
Let $G$ be a functional on $L^2(\Tn,\mathbb{R}^2)$ defined by (\ref{defG}), and let 
\begin{equation}
\label{scheme_formula1}
z = (I+\lambda \partial_{L^2} G)^{-1} y
\end{equation}
for some $y\in L^2(\Tn,\mathbb{R}^d)$ and $\lambda>0$. Then, we have that
$z = P_Z^{L^2}(y)$, where  $P_Z^{L^2}(y)$ denotes the orthogonal projection of $y$ on the set $Z$ with respect to the inner product in $L^2$.
\end{lemma}
\begin{proof}
The formula (\ref{scheme_formula1}) implies that $z$ is a solution to the equation
$$z+\lambda \partial_{L^2} G(z)\ni y\,,$$
which is the optimality condition to the problem
$$\inf_{z\in L^2} \left(\frac{1}{2}\norm{z-y}^2_{L^2} + \lambda G(z) \right) = \inf_{z\in Z} \frac{1}{2}\norm{z-y}^2_{L^2}\,.$$
This implies that $z = P_Z^{L^2}(y)$.
\end{proof}

\begin{remark}
\label{scheme_remark1}
Due to the form of the set $Z$, the~projection operator $P^{L^2}_Z$ is given  by the explicit formula
\begin{equation}
\label{proj_formula}
P^{L^2}_Z(y)= 
\dfrac{y}{\abs{y} \vee 1}\,,
\end{equation}
where $|\cdot|$ denotes the Euclidean norm and $a \vee b:=\max(a,b)$. Using this fact, Lemma \ref{scheme_lemma1} and Lemma \ref{scheme_lemma2}, we obtain an explicit form of the scheme (\ref{scheme3}). It is given by
\begin{equation}
\label{scheme2}
\left\{\begin{array}{ll} 
w^{k} = -\nabla(f+\tau(-\Delta)^s\div z^{k})\,, \\[0.2cm]
z^{k+1} = \dfrac{z^{k}-\lambda w^k}{|z^{k}-\lambda w^k| \vee 1}\,.
 \end{array}\right.
\end{equation}
\end{remark}

\section{Ergodic convergence}

In this section, we investigate ergodic convergence of sequences $\{v^k\}$ and $\{z^k\}$ defined by schemes (\ref{scheme1}) and (\ref{scheme2}), respectively.
First, we recall from Proposition \ref{prop_converg_vk} that 
if $0<\lambda\tau<2$, then the sequence $\{v^k\}$ converges weakly in $\Hsavneg(\Tn)$ to $v^\ast\in K$, where $v^\ast$ is a unique solution of the dual problem (\ref{dual_problem_v}).
Then, Mazur's lemma implies existence of the sequence
\begin{equation}
\label{barvn}
\bar{v}^n = \sum_{k=0}^n \alpha_k v^k\,,
\end{equation}
where $\{\alpha_k\}$ is such that $\sum_{k=0}^n \alpha_k = 1$, which converges strongly in $\Hsavneg(\Tn)$  to $v^\ast$ as $n\rightarrow \infty$. Our aim in this section is to construct a sequence $\{\alpha_k\}$ such that $\bar{v}^n\rightarrow v^\ast$ as $n\rightarrow\infty$, and next, to use this result in order to prove that  sequence
\begin{equation}
\label{barzn}
\bar{z}^n = \sum_{k=0}^n \alpha_k z^k\,,
\end{equation}
converges 
weakly in $QL^2(\Tn)$ to $z^\ast\in Z$ as $n\rightarrow\infty$, where $Q$ is the orthogonal projection onto the space of gradient fields. It is based on Helmholtz-Weyl decomposition. 
Note that for this proof, we will not require a weak convergence of $z^k$ to $z^\ast\in Z$ in $L^2(\Tn)$. 
We are not able to prove this in an infinite dimensional space setting
(see Lemma~\ref{lem_opnorm} and Proposition~\ref{prop_converg_zk}). 
All these results are given in the following proposition:

\begin{proposition}
Let $\{v^k\}$ be a weakly convergent sequence in $K$ generated by the scheme~(\ref{scheme1}) and let 
$\{\beta_k\}$ be a sequence of positive real numbers
such that $\{\beta_k\}\in l^2\setminus l^1$. Then, for $\alpha_k = \beta_k/\sum_{j=1}^n \beta_j$ 
the sequence $\{\bar{v}^n\}$ given by (\ref{barvn})   
converges strongly in $\Hsavneg(\Tn)$ to $v^\ast\in K$ as $n\rightarrow\infty$. Moreover, the sequence 
$\{\bar{z}^n\}$ given by (\ref{barzn}), where 
$\{z^k\}$ is generated by the scheme (\ref{scheme2}), converges weakly in $QL^2(\Tn)$ to $z^\ast\in Z$ as $n\rightarrow\infty$.
\end{proposition}

\begin{proof}
Let $\alpha_k = \beta_k/\sum_{k=0}^{n} \beta_k$, where the sequence $\{\beta_k\}\in l^2\setminus l^1$.
Then, by the triangle inequality and Cauchy-Schwarz inequality, we have
\begin{equation}
\label{converg_z_eq1}
\begin{split}
&\|\bar{v}^n-v^\ast\|_{\Hsavneg} 
=
\|\sum_{k=0}^n \alpha_k (v^k-v^\ast)\|_{\Hsavneg} 
\leq \sum_{k=0}^n \| \alpha_k (v^k-v^\ast)\|_{\Hsavneg}\\
&= \sum_{k=0}^n |\alpha_k| \| v^k-v^\ast\|_{\Hsavneg}
\leq \frac{1}{\sum_{k=0}^{n} \beta_k}\left(\sum_{k=0}^n \beta_k^2\right)^{1/2} \left(\sum_{k=0}^n \| v^k-v^\ast\|^2_{\Hsavneg}\right)^{1/2}\,.
\end{split}
\end{equation}
Since  $\{ \beta_k\}\in l^2\setminus l^1$, and therefore $\sum_{k=0}^{n} \beta_k\rightarrow\infty$ as $n\rightarrow \infty$, it remains to show that $\sum_{k=0}^n \| v^k-v^\ast\|^2_{\Hsavneg}$ is bounded.
Using the scheme (\ref{scheme1}) and the fact that
the resolvent $(I+\lambda \partial_{\Hsavneg} \Phi^\ast)^{-1}$ is non-expansive, we have
\begin{equation}
\begin{split}
&\| v^{k+1}-v^\ast\|_{\Hsavneg}^2\\
&=\| (I+\lambda \partial_{\Hsavneg} \Phi^\ast)^{-1}(v^k+\lambda u^k)-(I+\lambda \partial_{\Hsavneg} \Phi^\ast)^{-1}(v^\ast +\lambda u^\ast)\|_{\Hsavneg}^2\\
&\leq \| v^k-v^\ast + \lambda (u^k-u^\ast)\|_{\Hsavneg}^2
= |1-\lambda\tau|^2 \| v^k-v^\ast \|_{\Hsavneg}^2\,.
\end{split}
\end{equation}
Thus, we get,
$$
\| v^{k}-v^\ast\|_{\Hsavneg}^2
\leq  |1-\lambda\tau|^{2k} \| v^0-v^\ast \|_{\Hsavneg}^2
$$
for $k = 0,1,\ldots,n$.
By assumption, $0<\lambda\tau<2$, we have that $|1-\lambda\tau|<1$.
Therefore, we get the estimate
$$\sum_{k=0}^n \| v^k-v^\ast\|^2_{\Hsavneg}
\leq \| v^0-v^\ast \|_{\Hsavneg}^2 \sum_{k=0}^n |1-\lambda\tau|^{2k} <\infty\,.
$$
If we apply this fact to (\ref{converg_z_eq1}), then passing to the limit and using the monotone convergence theorem, we obtain $\lim_{n\rightarrow \infty} \|\bar{v}^n-v^\ast\|_{\Hsavneg}  = 0$. 

Due to the weak closedness of the set $K$, we have 
existence of $z^\ast\in Z$ such that $v^\ast= -\Dsav\div z^\ast$. We also have
\begin{equation}
\begin{split}
\bar{v}^n = \sum_{k=0}^n \alpha_k v^k 
&= -\sum_{k=0}^n \alpha_k \Dsav\div z^k\\
&= -\Dsav\div \left( \sum_{k=0}^n \alpha_k z^k\right)
=-\Dsav\div \bar{z}^n\,.
\end{split}
\nonumber
\end{equation}

Since
\begin{equation}
\begin{split}
\|\bar{v}^n-v^\ast\|_{\Hsavneg} 
&= \langle\Dsavneg(\bar{v}^n-v^\ast),\bar{v}^n-v^\ast\rangle\\
&= \langle\div(\bar{z}^n-z^\ast),\Dsav\div(\bar{z}^n-z^\ast)\rangle\\
&=\|\div(\bar{z}^n-z^\ast)\|_{\Hsav}\,,
\end{split}
\nonumber
\end{equation}
thus $\div\bar{z}^n
\to \div z^\ast$ in $\Hsav(\Tn)$.
Moreover, since
\begin{equation}
\begin{split}
&\left\lvert \int_{\Tn} (\bar{z}^n-z^\ast)\nabla \phi \,dx \right\rvert 
= \left\lvert \int_{\Tn} \div(\bar{z}^n-z^\ast)\phi \,dx \right\rvert\\
&\qquad\leq\|\div(\bar{z}^n-z^\ast)\|_{L^2} \|\phi\|_{L^2}
\leq\|\div(\bar{z}^n-z^\ast)\|_{\Hsav} \|\phi\|_{L^2}
\end{split}
\nonumber
\end{equation}
for all $\phi\in H_{\text{av}}^1(\Tn)$, we have
$\bar{z}^n
\rightharpoonup z^\ast$ in $QL^2(\Tn)$.
\end{proof}

\section{Implementation}

In this section, we explain how to discretize the system (\ref{scheme2}) and to solve numerically one iteration of the semi-implicit scheme (\ref{semidiscrete}). Recursive application of this procedure leads to a numerical solution of the nonlinear evolution equation (\ref{eq_rigorous}). For simplicity of presentation, we construct a method for the one-dimensional problem, but it can be easily extended to higher dimensions. We also present results concerning the convergence of the introduced scheme in a~finite dimensional space. 

We denote by $X$ the Euclidean space $\mathbb{R}^N$. The~scalar product of two elements $u$, $v\in X$ is defined by $\langle u, v\rangle := \sum_{i=1}^N u_i v_i$ and the norm $\|u\|:=\sqrt{\langle u, u\rangle}$.
The discrete gradient operator satisfying periodic boundary conditions is defined~as
$$\nabla: =\dfrac{1}{h}\left(\begin{array}{rrcrr}
-1 & 1 & \cdots & 0 & 0 \\ 
0 & -1 & \cdots & 0 & 0 \\ 
\vdots & \vdots & \ddots & \vdots & \vdots \\ 
0 & 0 & \cdots & -1 & 1 \\
1 & 0 & \cdots & 0 & -1 
\end{array} \right)_{N\times N}\,,$$
where $h>0$ is the space discretization step.
With this notation in hand, the discrete divergence and the Laplace operator are given by $\div :=\nabla^T$ and
$\Delta := \div \nabla$, respectively. Note that $\Delta$ is a symmetric, positive definite, circulant matrix of size $N\times N$.
It is well know, that in this case, its eigenvectors are given by
$$\nu^j = \frac{1}{\sqrt{N}}\left(1,e^{2\pi i j/N},\ldots,e^{-2\pi i j (N-1)/N}\right)^T$$
for $j = 0,\ldots,N-1$, and corresponding eigenvalues are
$$\mu_j = \sum_{k=0}^{N-1} c_k e^{2\pi i j k/N}\,,$$
where $c_k$, $k=0,\ldots,N-1$, are subsequent elements of the first row of $\Delta$. Therefore, analogously to the continuous setting, we can define the operator $(-\Delta)^s$ using the discrete Fourier transform, i.e.,
\begin{equation}
\label{def_fraclap_discrete}
\widehat{(-\Delta)^s u}:=M^{s}\widehat{u}
\end{equation}
where $M^s$ is an $N\times N$ diagonal matrix with elements $\mu_j^s$, $j=0,\ldots,N-1$.

Here $\hat{u}$ denotes the discrete Fourier transform of $u\in X$, i.e., 
$$\hat{u}_k = \sum_{j=0}^{N-1} u_{j} e^{-2\pi i j k/N}$$
for $k=0,\ldots,N-1$. The discrete inverse Fourier transform is given by
$$u_j = \frac{1}{N}\sum_{k=0}^{N-1} \hat{u}_{k} e^{2\pi i j k/N}$$
for $j=0,\ldots,N-1$.

For convenience, we also introduce the definition of scalar product 
\begin{equation}
\label{scalar_prod_discrete}
(u,v)_{-s}:=\langle (-\Delta)^{-s} u, v\rangle
\end{equation}
for all $u$, $v\in X$ and the norm $\norm{u}_{-s}:=\sqrt{(u,u)_{-s}}$.

\begin{lemma}
\label{lem_opnorm}
For $z\in X$, there exists a constant $C>0$, such that
\begin{equation}
\norm{(-\Delta)^s\div z}^2_{-s}\leq C\norm{z}^2\,.
\end{equation}
Moreover, $C = \mu_{\text{max}}^{s+1}$, where $\mu_{\text{max}}$ denotes the largest eigenvalue of the discrete Laplace operator.
\end{lemma}
\begin{proof}
From definitions (\ref{def_fraclap_discrete}), (\ref{scalar_prod_discrete}) and Parseval's theorem, we have
\begin{equation}
\label{estim1}
\norm{(-\Delta)^s \div z}^2_{-s} = \frac{1}{N}\langle M^s\, \widehat{\div z},\widehat{\div z}\rangle\,.
\end{equation}
Since,
$$|\widehat{\div z}|^2
= M |\widehat{z}|^2\,,$$
then using this fact in (\ref{estim1}) and applying the Plancherel identity, we get
\begin{equation}
\begin{split}
\norm{\Dsav \div z}^2_{-s} 
& 
\leq \frac{1}{N} \norm{M^s} 
\|\widehat{\div z}\|^2\\
& \leq \max_{j} \,\mu_j^{s+1} \, \frac{1}{N}\|\widehat{z}\|^2
=  \max_{j} \,\mu_j^{s+1} \, \norm{z}^2\,.
\end{split}
\nonumber
\end{equation}
\end{proof}

We denote by $Z:=\{ z\in X \ : \ \norm{z}_\infty\leq 1 \}$, where $\norm{z}_\infty : = \max_j |z|$. 
For $f\in X$ we define functionals $F$ and $G$ on $X$ by
\begin{equation}
\label{defF2}
 F(z):=\left\{\begin{array}{ll} 
\dfrac{1}{2\tau}\norm{\tau (-\Delta)^s \div z + f}_{-s}^2 & \text{ if } z\in Z\,,\\[0.2cm]
+\infty & \text{ otherwise}\,.
 \end{array}\right.
\end{equation}
and 
\begin{equation}
\label{defG2}
G(z) := \left\{\begin{array}{ll} 
 0 & \text{ if } z\in Z\,,\\[0.1cm]
+\infty & \text{ otherwise}\,.
\end{array}\right.
\end{equation}
Then the discrete version of the problem (\ref{min_problem_z_cont}) is 
$$\min_{z\in X} \left( F(z) + G(z) \right)$$
and can be solved by the iterative scheme
\begin{equation}
\label{scheme_for_z_discrete}
\left\{\begin{array}{l} 
w^k \in \partial_X F(z^k)\,, \\[0.2cm]
z^{k+1} = (I+\lambda \partial_X G)^{-1}(z^k - \lambda w^k)\,.
\end{array}\right.
\end{equation}

\begin{proposition}
\label{prop_converg_zk}
Assume that $C \lambda\tau <2$, where the constant $C>0$ is as in Lemma \ref{lem_opnorm}, then $z^k\rightharpoonup z^\ast$  in $X$, where $z^\ast\in Z$ is such that $0\in \partial_X F(z^\ast) + \partial_X G(z^\ast)$.
\end{proposition}
\begin{proof}
In the proof,  we essentially follow steps in the proof of Proposition \ref{prop_converg_vk}.
We assume that $z^{k}$, $z^{k+1} \in Z$. After simple calculations, we obtain
\begin{equation}
\label{proof_calc1}
\begin{split}
&F(z^{k+1})-F(z^k)\\
&= \frac{1}{2\tau}\norm{\tau (-\Delta)^s \div z^{k+1}+f}_{-s}^2
-\frac{1}{2\tau}\norm{\tau (-\Delta)^s \div z^{k}+f}_{-s}^2\\
&=\frac{\tau}{2}\norm{(-\Delta)^s\div (z^{k+1}-z^{k})}_{-s}^2
 + \langle f+\tau (-\Delta)^s\div z^{k}, \div (z^{k+1}-z^{k})\rangle
\end{split}
\end{equation}

From the second line of the scheme (\ref{scheme_for_z_discrete}) we have that
$$\frac{1}{\lambda} (z^k-z^{k+1})-w^k \in \partial_X G(z^{k+1})$$
Using the definition of subdifferential $\partial_X G$ and the fact that $G$ is an indicator function of the set $Z$, we get
$$ \frac{1}{\lambda}  \langle z^k-z^{k+1}-\lambda w^k , p- z^{k+1} \rangle\geq 0$$
for any $p\in Z$. Then in particular for $p = z^k$ and using that 
$w^{k} = -\nabla(f+\tau(-\Delta)^s\div z^{k})$,
we obtain
\begin{equation*}
-\langle \tau(-\Delta)^s \div z^k + f,\div(z^{k}-z^{k+1})\rangle\leq -\frac{1}{\lambda}\norm{z^k-z^{k+1}}^2\,.
\nonumber
\end{equation*}
Taking into account this fact and Lemma \ref{lem_opnorm} in (\ref{proof_calc1}), we obtain
\begin{equation}
\begin{split}
F(z^{k+1})-F(z^k)
&\leq \dfrac{\tau}{2}  \norm{(-\Delta)^s\div (z^{k+1}-z^k)}_{-s}^2-\dfrac{1}{\lambda}\norm{z^k-z^{k+1}}^2\\
&\leq \left(\dfrac{\tau}{2} C -\dfrac{1}{\lambda}\right) \norm{z^k-z^{k+1}}^2\,.
\end{split}
\end{equation}
Therefore, taking $\lambda$ so that $C\lambda\tau<2$, we get 
$F(z^{k+1})<F(z^k)$, as long as $z^{k+1}\neq z^k$, which implies that the sequence $\{F(z^k)\}$ is convergent.
To complete the proof we proceed in a similar way as in the proof of Proposition~\ref{prop_converg_vk} using the fact that in a finite dimensional space linear operators are bounded and norms are equivalent.
\end{proof}

We note that from Proposition \ref{prop_converg_zk} and Lemma \ref{lem_opnorm} it follows that the sequence $\{z^k\}$ converges weakly to $z^\ast$ in $X $ if and only if $$C\tau\lambda<2\,,$$ 
where the constant  $C = \mu_{\textrm{max}}^{s+1}$ and $\mu_{\textrm{max}}:=\max_j \mu_j$ is the largest eigenvalue of the discrete Laplace operator $\Delta$. In the case of considered here finite difference discretization, we have that 
$$
\mu_{\textrm{max}} = \frac{1}{h^2}\left\{
\begin{array}{cc}
4 & \text{ if $N$ is even}\,, \\[0.2cm]
4 \sin((N-1)\pi/2N) & \text{ if $N$ is odd}\,,
\end{array} 
\right.
$$
where $h$ denotes the space discretization step. 
In other words, to achieve the convergence one has to select values of parameters $\tau$, $\lambda$ and $h$ so that $\mu_{\text{max}}^{s+1}\tau\lambda<2$. Taking, for instance, $\tau = O(h^{2s+3})$, $\lambda = O(1/h)$ and passing to the limit with $h\rightarrow 0$, we get the desired result. 

\section{Numerical results}

In this section, we present results of numerical experiments, obtained by application of the introduced earlier schemes, to solve the evolution equation~(\ref{eq_rigorous}). 
These experiments have been performed on one-dimensional data for more accurate presentation of differences in solutions with respect to different values of the index $s$ and their characteristic features. 

For experiments, we were considering initial data $f$, $g:[-10,10]\rightarrow \mathbb{R}$, given by explicit formulas
\[ f(x) = \left\{
  \begin{array}{c l}
     20& \quad \text{if } |x|\leq 2 \\[0.2cm]
    50\,|x|^{-1}-5 & \quad \text{otherwise}
    \end{array} \right.,\qquad  
  g(x) = \left\{
  \begin{array}{c l}
    20& \quad \text{if } |x|\leq 2 \\[0.2cm]
    0 & \quad \text{otherwise } 
  \end{array} \,.\right.\]

\begin{figure}[h!]
\centering
\includegraphics[scale=0.3]{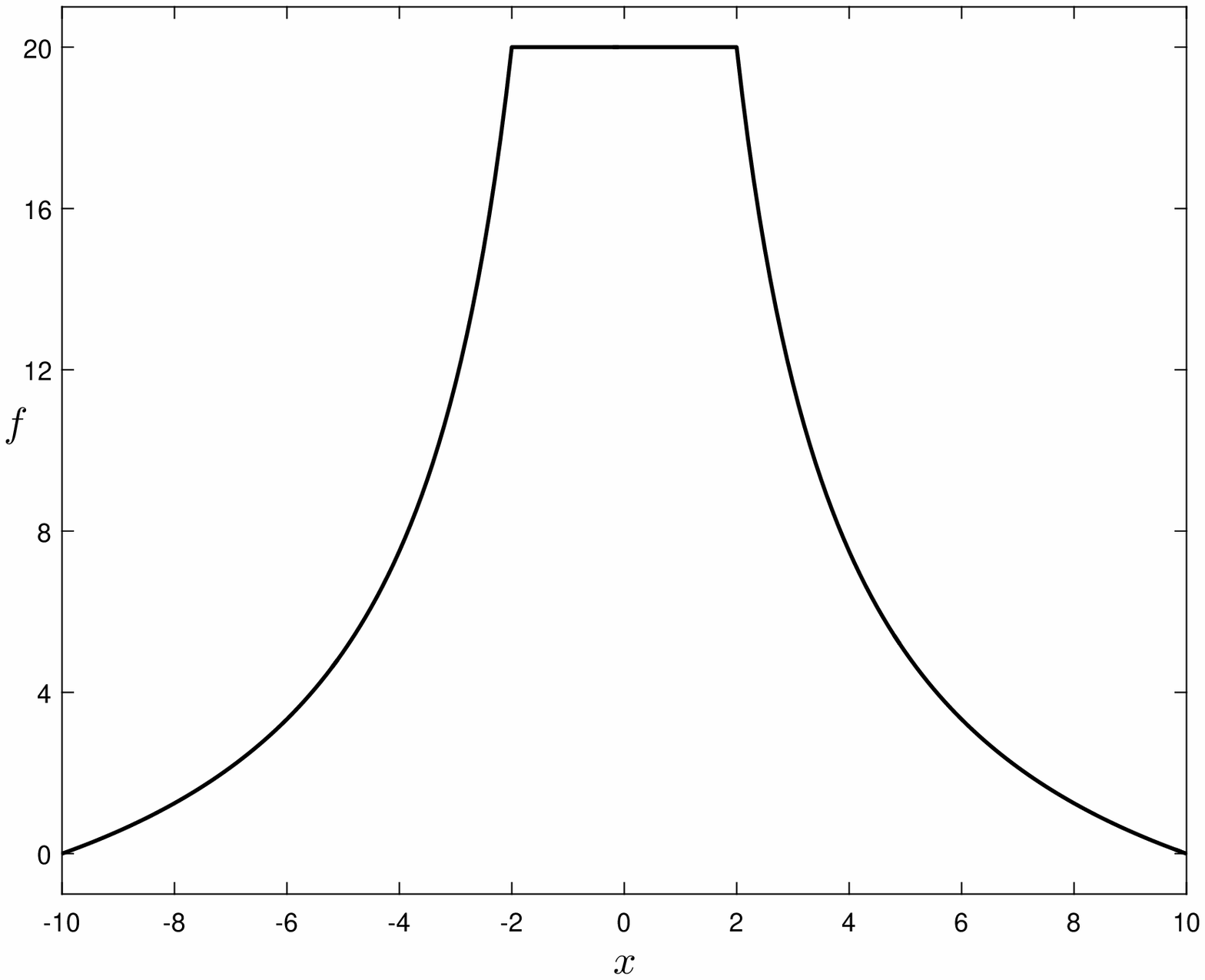}\hspace{0.2cm}
\includegraphics[scale=0.3]{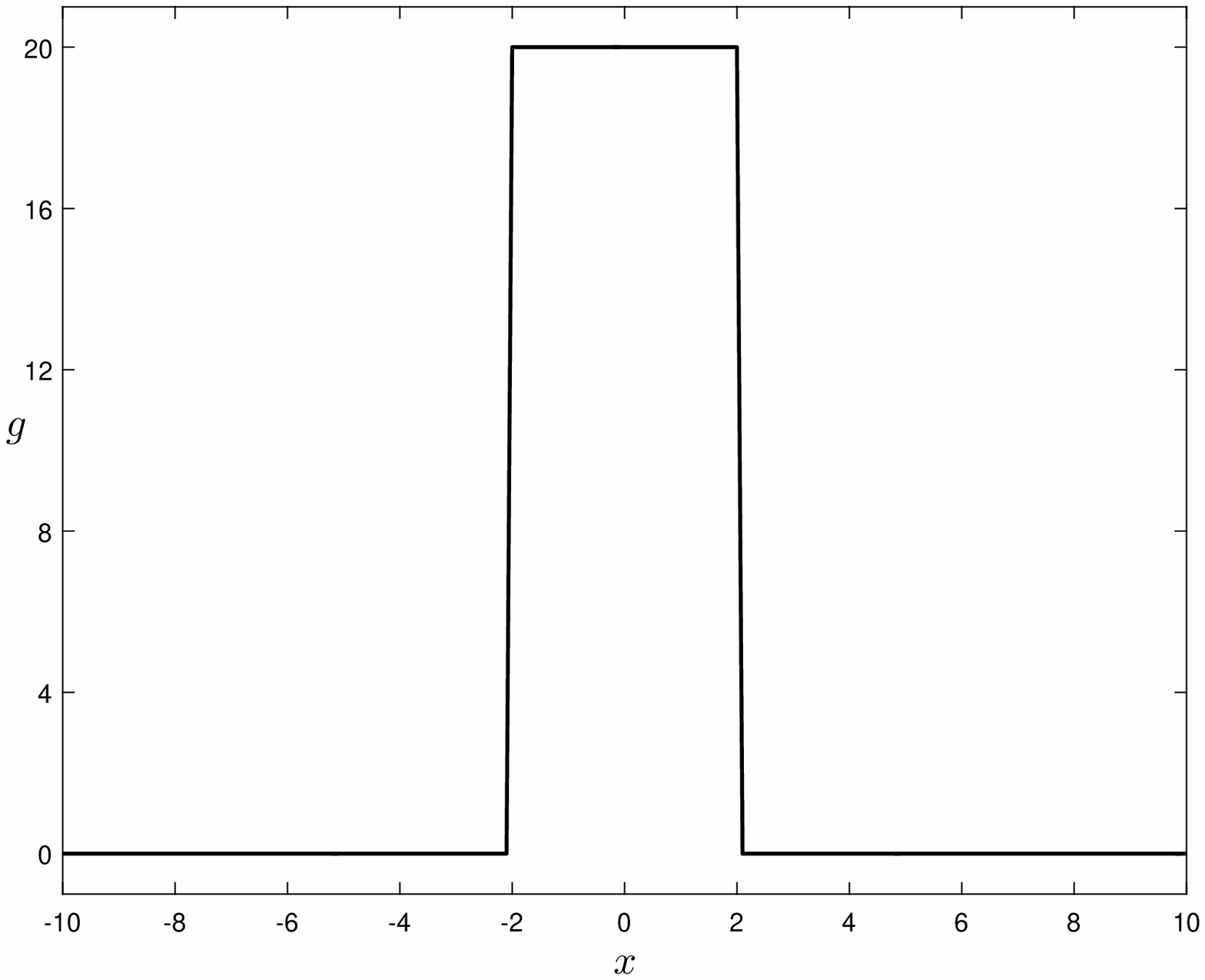}
\caption{Graphs of functions $f$ and $g$ considered in experiments as initial data.}
\label{fig1}
\end{figure}

In all calculations, we have fixed values of  parameters $h = 0.1$ and $\tau = 0.1$. In each case, a value of $\lambda$ was selected so that 
the criterion for the convergence given in Proposition \ref{prop_converg_zk} would be satisfied, i.e., we set 
$\lambda = 4\cdot 10^{-2}$,
$\lambda = 2\cdot 10^{-3}$ and
$\lambda =  10^{-4}$ for 
$s = 0$, $s=0.5$ and $s=1$, respectively. In Figure \ref{fig2}, we present evolution of solutions to the $H^{-s}$ total variation flow for initial data $f$ and $g$, and for different values of the index $s$ ($s=0, 0.5, 1$).

\begin{figure}[h!]
\centering
\includegraphics[scale=0.3]{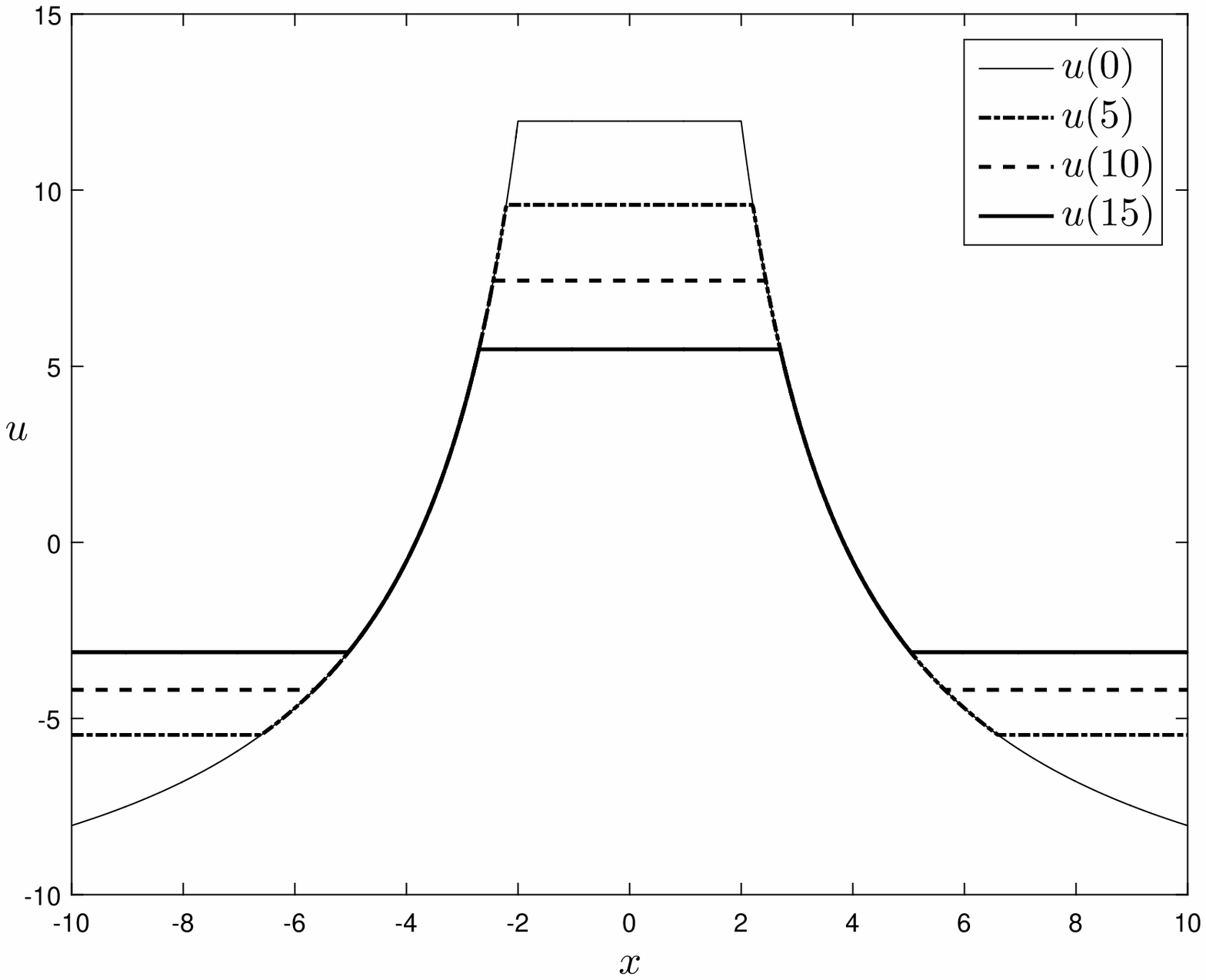}\hspace{0.2cm}
\includegraphics[scale=0.3]{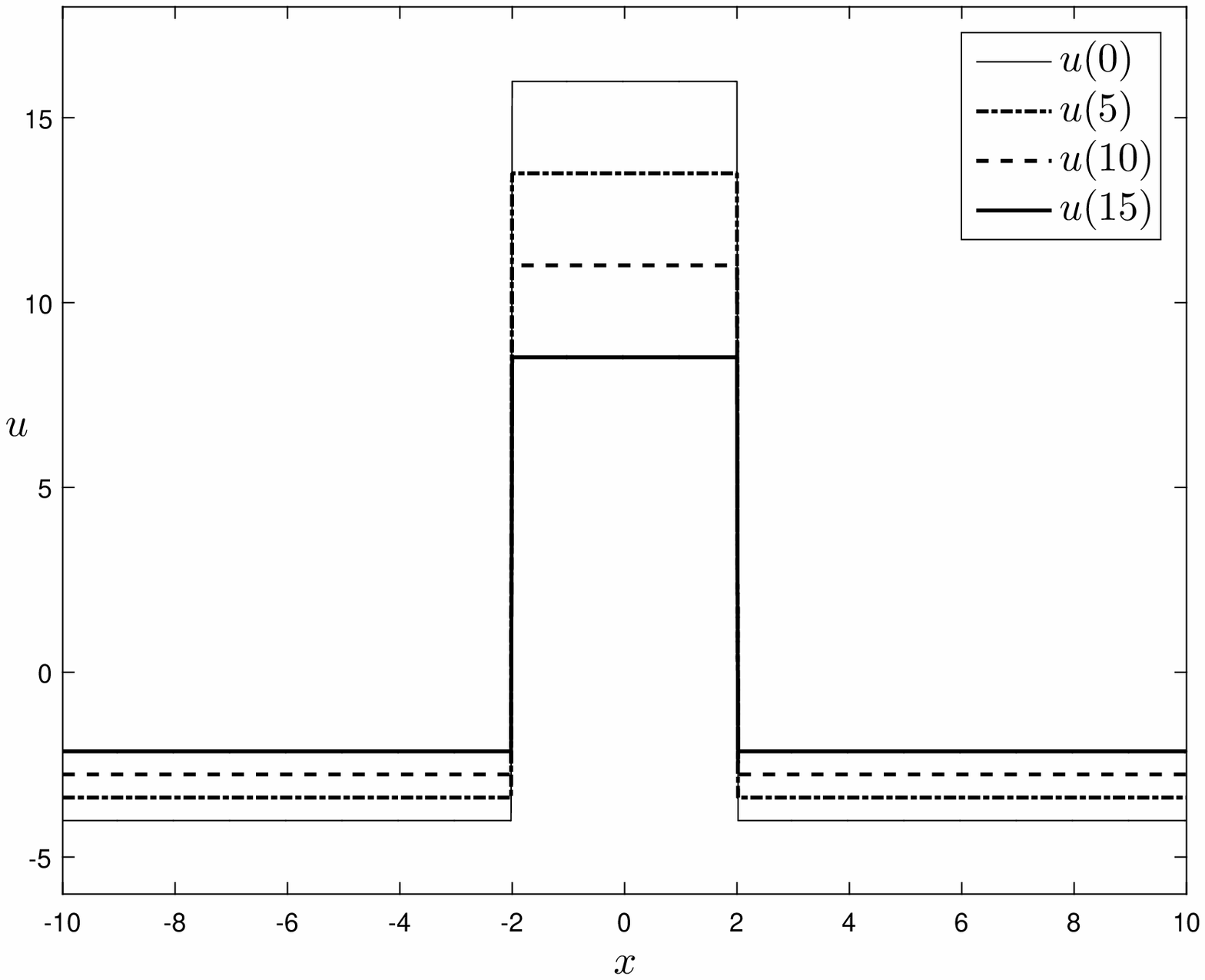}\\[-0.2cm]
$s = 0$\\[0.5cm]
\includegraphics[scale=0.3]{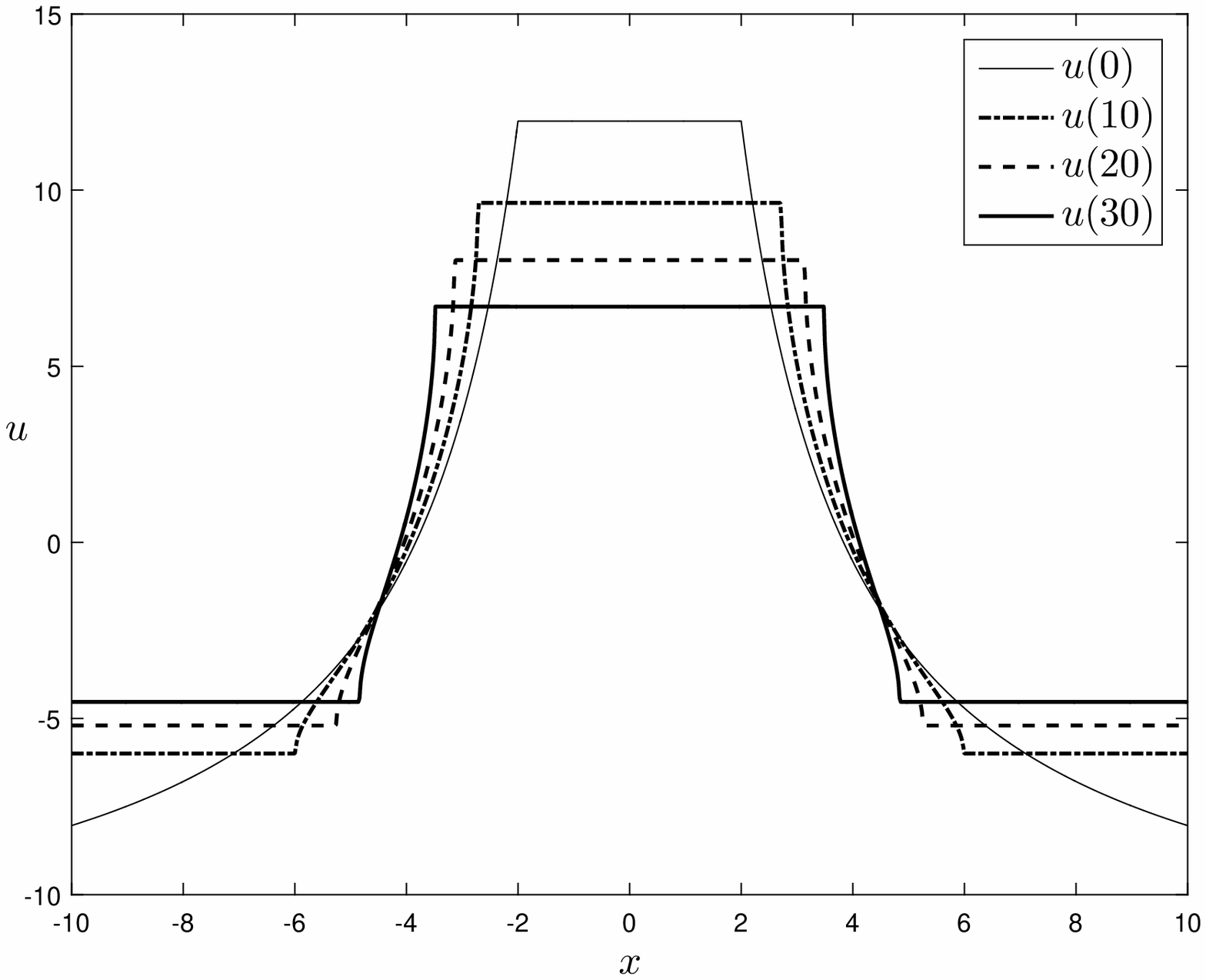}
\hspace{0.2cm}
\includegraphics[scale=0.3]{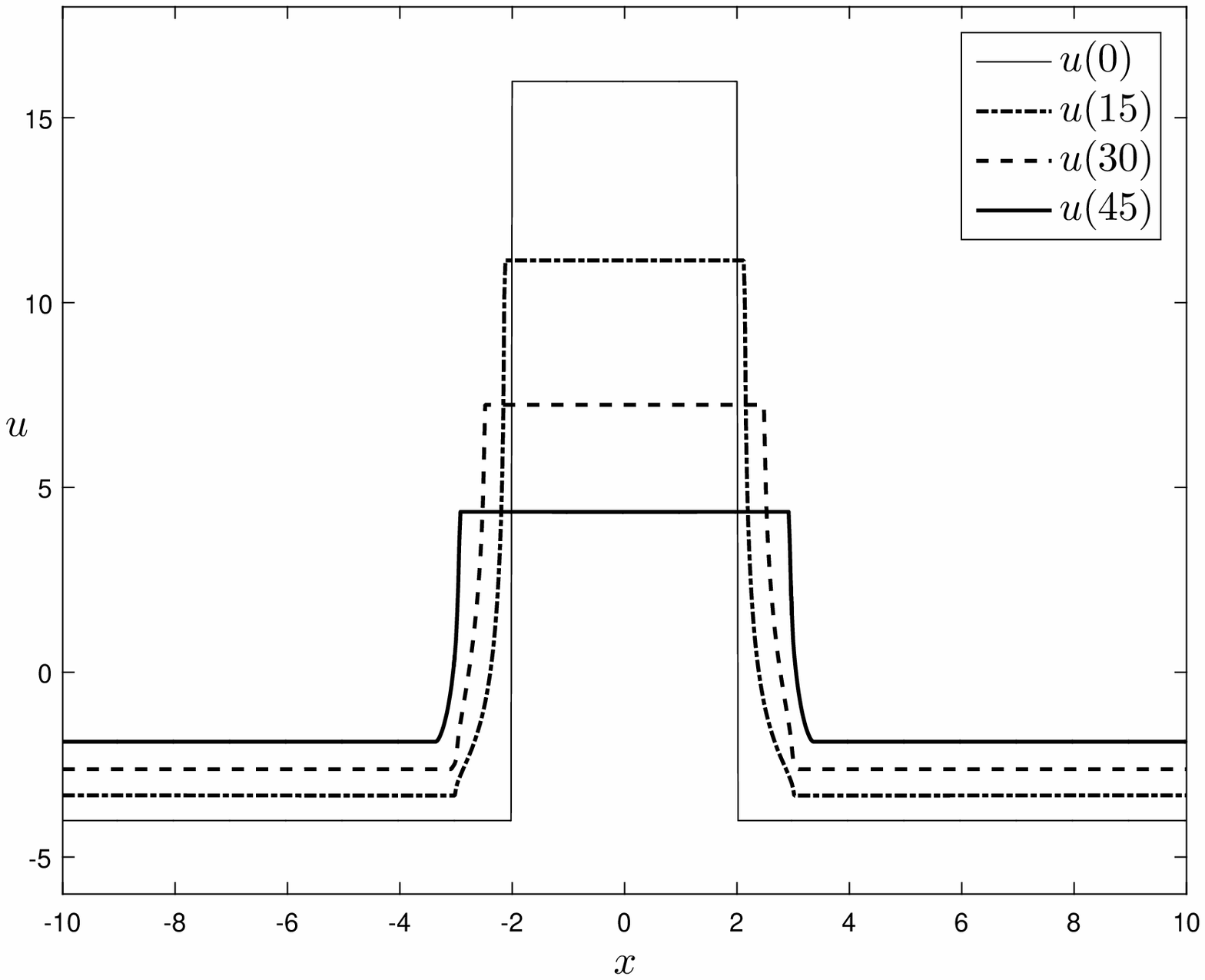}\\[-0.2cm]
$s = 0.5$\\[0.5cm]
\includegraphics[scale=0.3]{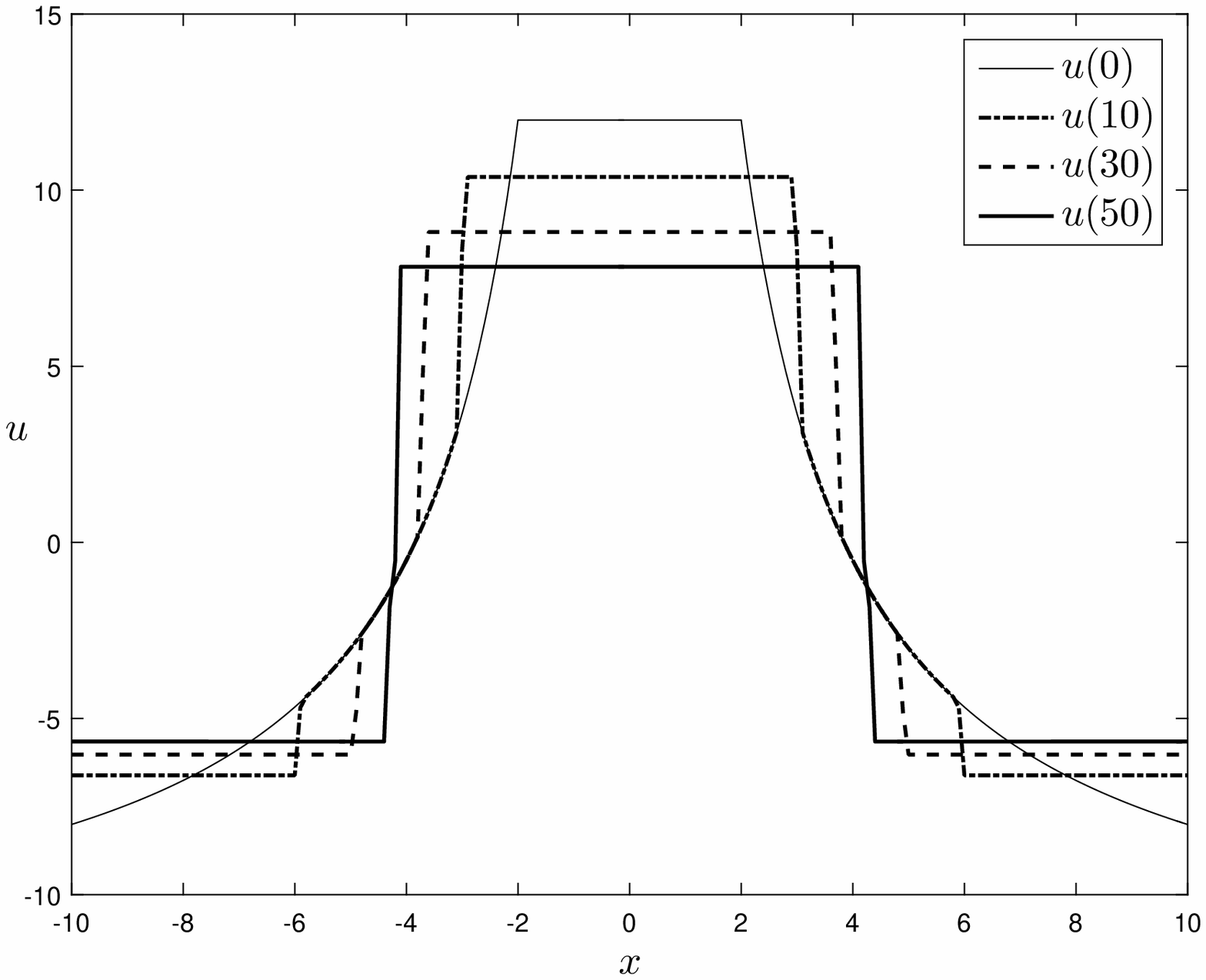}\hspace{0.2cm}
\includegraphics[scale=0.3]{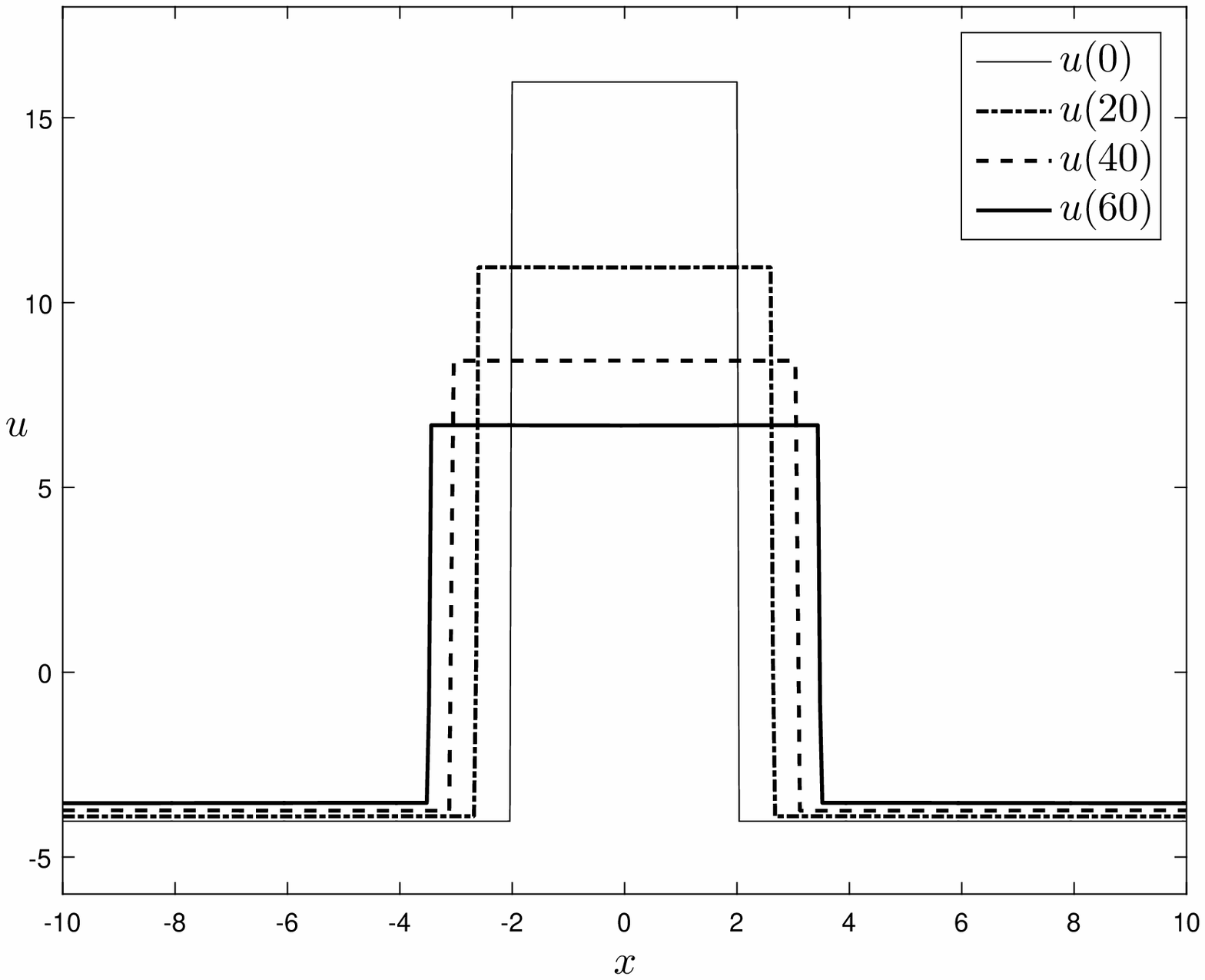}\\[-0.2cm]
$s = 1$
\caption{Evolution of solutions to the $H^{-s}$ total variation flow for initial data $f$ and $g$, and for different values of the index $s$.}
\label{fig2}
\end{figure}

From our computation, we conjecture that a solution may be instantaneously discontinuous for $s \in (0,1]$ for Lipschitz initial data. This is rigorously proved for $s=1$ in \cite{GigGig10}. Also, we see from this computation, that the motion becomes slower as $s$ becomes larger.



\section*{Acknowledgement}
This work was partially supported by the EU IRSES program ``FLUX'' and the Polish Ministry of the Science and Higher Education  grant number 2853/7.PR/2013/2.
The work of the first author was partially supported by Japan Society for the Promotion of Science through grants No. 26220702 (Kiban S) and No. 16H03948 (Kiban B). A part of the research for this paper was performed, when the second and the third authors visited the University of Tokyo. Its hospitality is gratefully acknowledged.


\begin{thebibliography}{1}

\bibitem{Aujol2009}
J.-F. Aujol,
\newblock Some first-order algorithms for total variation based image restoration, 
\newblock \emph{J. Math. Imag. Vis.}, 34(3), 307--327, 2009.

\bibitem{AnBaCaMa2001}
F. Andreu, C. Ballester, V. Caselles and J. M. Maz\'on, 
\newblock Minimizing Total Variation Flow, 
\newblock \emph{Diff. and Int. Eq.}, 14:321--360, 2001.

\bibitem{AndMazCas04}
F. Andreu-Vaillo, J. M. Maz\'on, V. Caselles,
\newblock \emph{Parabolic Quasilinear Equations Minimizing Linear Growth Functionals,}
\newblock Progress in Mathematics, 223, Birkhauser Verlag, Basel, 2004.


\bibitem{Beck2009}
A. Beck and M. Teboulle, 
\newblock A fast iterative shrinkage-thresholding algorithm for linear inverse problems. 
\newblock {\it SIAM J. Imaging Sciences} 2(1): 183--202, 2009.

\bibitem{BeCaNo2002}
G. Bellettini, V. Caselles, M. Novaga, 
\newblock The total variation flow in $\mathbb{R}^N$, 
\newblock \emph{J. Differential Equations}, 184(2):475--525, 2002.

\bibitem{Ber73}
H. Brezis,
\newblock \emph{Op\'erateurs maximaux monotones et semi-groupes de contractions dans les espaces de Hilbert}, 
North Holland Publishing Company, Amsterdam, 1973.

\bibitem{BurHeSch09}
M. Burger, L. He, C.-B. Schoenlieb, 
Cahn-Hilliard inpainting and a generalization for gray-value images, \newblock \emph{SIAM J. Imaging Sci.}, 2(4):1129--1167, 2009.

\bibitem{Chambolle2004}
A.~Chambolle,
\newblock An algorithm for total variation minimization and applications, 
\newblock \emph{J. Math. Imag. Vis.}, 20(1-2),89--97, 2004.

\bibitem{Pock2011}
A. Chambolle and T. Pock, 
\newblock A first-order primal-dual algorithm for convex problems with applications to imaging. 
\newblock {\it Journal of Mathematical Imaging and Vision} 40, 120--145, 2011.

\bibitem{CraLig71}
M. G. Crandall, T. M. Liggett,
\newblock Generation of semi-groups of nonlinear transformations on general Banach spaces
\newblock \emph{Amer. J. Math.}, 93:265--298, 1971.

\bibitem{Ekeland1999}
I. Ekeland and R. T\'emam, 
\newblock \emph{Convex analysis and variational problems}, 
\newblock volume 28 of Classics in Applied Mathematics. SIAM, 1999.

\bibitem{EllSmi2007}
C. M. Elliott and S. A. Smitheman,
\newblock Analysis of the TV regularization and $H^{-1}$ fidelity model for
decomposing an image into cartoon plus texture,
\newblock \emph{Commun. 
Pure  Appl. Anal.}, 6(4):917--936, 2007.

\bibitem{EllSmi2008}
C. M. Elliott and S. A. Smitheman,
\newblock Numerical analysis of the TV regularization and $H^{-1}$ fidelity model for
decomposing an image into cartoon plus texture,
\newblock \emph{IMA Journal of Numerical Analysis}, 1--39, 2008.


\bibitem{GigGig10}
M.-H. Giga and Y. Giga, 
\newblock Very singular diffusion equations-second and fourth order
problems, 
\newblock \emph{Japanese J. Appl. Math.}, 27:323--345, 2010.

\bibitem{GGP}
M.-H. Giga, Y. Giga and N. Pozar,
\newblock Periodic total variation flow of non-divergence type in $\mathbb{R}^n$, 
\newblock{\it J. Math. Pures Appl}, 102:203--233, 2014.

\bibitem{GigKoh11}
Y. Giga and R. V. Kohn, 
\newblock Scale-invariant extinction time estimates for some singular
diffusion equations, 
\newblock \emph{Discrete Contin. Dyn. Syst.}, 30(2):509--535, 2011.

\bibitem{GigKurMat14}
Y. Giga,  H. Kuroda, H. Matsuoka,
\newblock Fourth-order total variation flow with Dirichlet condition: characterization of evolution and extinction time estimates, 
\newblock \emph{Adv. Math. Sci. App.}, 24:499--534, 2014.

\bibitem{K2004} 
Y. Kashima,
\newblock A subdifferential formulation of fourth order singular diffusion equations, 
\newblock{\it  Adv. Math. Sci. Appl.} 14:49--74, 2004.

\bibitem{Kas12}
Y. Kashima, 
\newblock Characterization of subdifferentials of a singular convex functional in Sobolev spaces of order minus one, 
\newblock \emph{J. Funct. Anal.}, 262(6):2833--2860, 2012.

\bibitem{KV}
R. V. Kohn and H. M. Versieux,
\newblock Numerical analysis of a steepest-descent PDE model for surface relaxation below the roughening temperature, 
\newblock{\it SIAM J. Numer. Anal.} 48:1781--1800, 2010.

\bibitem{Ko}
Y. Komura. 
\newblock Nonlinear semi-groups in Hilbert space,
\newblock {\it J. Math. Soc. Japan}, 19:493--507, 1967.

\bibitem{LioMer79}
P. L. Lions and B. Mercier,
\newblock Splitting algorithms for the sum of two nonlinear operators.
\newblock \emph{SIAM J. Numer. Anal.}, 16:964--979, 1979. 

\bibitem{LLMM}
J.-G. Liu, J. Lu, D. Margetis and J. L. Marzuola,
\newblock Asymmetry in crystal facet dynamics of homoepitaxy by a continuum model, \newblock{\it preprint}. 

\bibitem{LMM}
M. \L{}asica, S. Moll and P. B. Mucha,
\newblock Total variation denoising in $l^1$ anisotropy.
\newblock arXiv:1611.03261, 2016.

\bibitem{Meyer2001}
Y. Meyer,
\newblock \emph{Oscillating Patterns in Image Processing and Nonlinear Evolution Equations: The Fifteenth Dean Jacqueline B. Lewis Memorial Lectures},
\newblock American Mathematical Society,
Boston, MA, USA, 2001.

\bibitem{Moll2005}
J. Moll,
\newblock The anisotropic total variation flow.
\newblock {\it Math. Ann.} 332:177--218, 2005.

\bibitem{MucMusRyb15}
P. B. Mucha, M. Muszkieta, P. Rybka,
\newblock Two cases of squares evolving by anisotropic diffusion,
\newblock{\it Advances in Differential Equations}, 20(7--8):773--800, 2015.

\bibitem{OOTT}
A. Oberman, S. Osher, R. Takei and R. Tsai, 
\newblock Numerical methods for anisotropic mean curvature flow based on a discrete time variational formulation, 
\newblock{\it Commun. Math. Sci.} 9:637--662, 2011.

\bibitem{OshSolVes2003}
S. J. Osher, A. Sole, L. A. Vese,
\newblock Image decomposition and restoration using total variation minimization and the $H^{-1}$ norm,
\newblock \emph{Multiscale Modeling and Simulation: A SIAM Interdisciplinary Journal}, 1(3):349--370, 2003. 

\bibitem{ROF1992}
L. Rudin, S. Osher and E. Fatemi,
\newblock Nonlinear total variation based noise removal algorithms,
\newblock \emph{Physica D}, 60:259--268, 1992.

\bibitem{Rul96}
J. Rulla,
\newblock  Error analysis for implicit approximations to solutions to Cauchy problems.
\emph{SIAM J. Numer. Anal.}, 33(1):68--87, 1996.

\bibitem{Scho09}
C.-B. Sch\"onlieb, 
\newblock Total variation minimization with an $H^{-1}$ constraint, 
\newblock \emph{CRM Series 9, Singularities in Nonlinear Evolution Phenomena and Applications proceedings, Scuola Normale Superiore Pisa}, 201--232, 2009.

\bibitem{Weiss2009}
P. Weiss, G. Aubert, L. Blanc-F\'eraud,
\newblock Efficient schemes for total
variation minimization under constraints in image processing.
\newblock {\it SIAM Journal on Scientific Computing},  31(3):2047--2080, 2009.



\end{thebibliography}
\end{document}